\documentclass[12pt]{amsart}
\usepackage{amssymb,amsmath,amsfonts,latexsym,comment}

\newcommand{\newsection}[1]{\setcounter{equation}{0} \section{#1}}
\numberwithin{equation}{section}

\newtheorem{theorem}{Theorem}[section]

\newtheorem{proposition}[theorem]{Proposition}
\newtheorem{lemma}[theorem]{Lemma}
\newtheorem{coro}[theorem]{Corollary}
\newtheorem{definition}[theorem]{Definition}

\newtheorem{remark}[theorem]{Remark}
\newtheorem{example}[theorem]{Example}
\newtheorem{question}{Question}

\def \D{\mathbb{D}}
\def \T{\mathbb{T}}

\newcommand{\clb}{\mathcal{B}}
\newcommand{\cle}{\mathcal{E}}
\newcommand{\clh}{\mathcal{H}}
\newcommand{\clw}{\mathcal{W}}
\newcommand{\clk}{\mathcal{K}}
\newcommand{\cls}{\mathcal{S}}

\newcommand{\clm}{\mathcal{M}}
\newcommand{\cll}{\mathcal{L}}
\newcommand{\vp}{\varphi}

\newcommand{\Z}{\mbox{$\mathbb Z$}}

\newcommand{\ran}{\mbox{ran}}

\begin{document}

\title[Compact and normal isometric pairs]{Isometric pairs with compact and normal cross-commutator}
	
\author[De]{Sandipan De}
\address{Sandipan De, School of Mathematics and Computer Science, Indian Institute of Technology Goa, Farmagudi, Ponda-403401, Goa, India.}
\email{sandipan@iitgoa.ac.in}

\author[Sarkar]{Jaydeb Sarkar}
\address{Jaydeb Sarkar, Indian Statistical Institute, Statistics and Mathematics Unit, 8th Mile, Mysore Road, Bangalore, 560 059, India.}
\email{jaydeb@gmail.com, jay@isibang.ac.in}
	
\author[Shankar]{P Shankar}
\address{Shankar. P, Department of Mathematics, Cochin University of Science and Technology, Kochi 682022, Kerala, India.}
\email{shankarsupy@cusat.ac.in, shankarsupy@gmail.com}
	
\author[Sankar]{Sankar T.R.}
\address{Sankar T.R., Department of General Sciences,  Birla Institute of Technology and Science Pilani, Dubai Campus, Dubai International Academic City, Dubai, UAE}
\email{sankartr90@gmail.com, sankar@dubai.bits-pilani.ac.in }


	
\begin{abstract}
We represent and classify pairs of commuting isometries $(V_1, V_2)$ acting on Hilbert spaces that satisfy the condition
\[
[V_1^*, V_2] = \text{compact and normal},
\]
where $[V_1^*, V_2] := V_1^* V_2 - V_2 V_1^*$ is the cross-commutator of $(V_1, V_2)$. The precise description of such pairs also gives a complete and concrete set of unitary invariants. The basic building blocks of representations of such pairs consist of four distinguished pairs of commuting isometries. One of them relies on some peculiar examples of invariant subspaces tracing back to Rudin's intricate constructions of analytic functions on the bidisc. Along the way, we present a rank formula for a general pair of commuting isometries that looks to be the first of its kind.
\end{abstract}

\maketitle

\tableofcontents

\newsection{Introduction}	

A linear operator $V$ on a Hilbert space $\clh$ (all our Hilbert spaces are complex and separable) is said to be an isometry if
\[
\|Vh\| = \|h\| \qquad (h \in \clh).
\]
Isometries are of essential importance in the field of linear analysis. In the context of infinite-dimensional spaces and even in the realm of the basic theory of linear operators, they serve as a building block for bounded linear operators as well as in the construction of elementary $C^*$-algebras. The structure of isometries is simple; they are either shift operators or unitary operators or direct sums of them. More specifically, the classical von Neumann-Wold decomposition theorem says \cite{von, Wold}: If $V$ is an isometry on $\clh$, then there exist unique $V$-reducing closed subspaces $\clh_u$ and $\clh_s$ of $\clh$ (one may be the zero space) such that
\begin{equation}\label{eqn WD 1}
\clh = \clh_u \oplus \clh_s,
\end{equation}
where $V|_{\clh_u}$ is a unitary and $V|_{\clh_s}$ is a shift. Therefore
\[
V = \begin{bmatrix}
\text{unitary} &  0 \\
0 & \text{shift}
\end{bmatrix},
\]
on $\clh = \clh_u \oplus \clh_s$. We recall that an isometry $V$ is said to be a \textit{shift} if $V$ does not have a unitary summand, equivalently
\begin{equation}\label{eqn def shift}
SOT-\lim_{n \rightarrow \infty} V^{*n} =0.
\end{equation}

It is now natural to look into the structure of pairs of commuting isometries. A \textit{pair of commuting isometries}, or simply an \textit{isometric pair}, refers to isometries $V_1$ and $V_2$ acting on some Hilbert space such that
\[
V_1 V_2 = V_2 V_1.
\]
We write this as $(V_1, V_2)$. Despite the fact that this programme was initiated a long time ago, progress on this problem has been rather sluggish, due in part to the vastly more convoluted structure of such objects. In terms of situations, we note that the $C^*$-algebras generated by isometric pairs are uncharted territory (however, see \cite{BCL1978, Ron, Jorgen, Muhly, Seto}). Second, the structure of isometric pairs would also disclose the complex and mysterious structure of shift-invariant subspaces of the Hardy space over the bidisc \cite[Theorem 3.1]{Yang-S}. Finally, isometric pairs represent all contractive linear operators on Hilbert spaces, which is a notoriously complex and year-old field of research. Because of this, it is desirable to set conditions on isometric pairs and look into the classification and computable invariants that can be found between them. In this paper, we identify a large class of isometric pairs and then represent and classify them in terms of concrete (or model) isometric pairs. The model also exhibits complete unitary invariants that are numerical in nature.

To demonstrate the conditions we will put on the isometric pairs under lookup, it is necessary to explain one of the simplest examples of isometric pairs, namely, doubly commuting pairs. A \textit{doubly commuting pair} is an isometric pair $(V_1, V_2)$ such that
\[
[V_1^*, V_2] = 0,
\]
where
\[
[V_1^*, V_2] := V_1^* V_2 - V_2 V_1^*,
\]
is the cross-commutator of the pair $(V_1, V_2)$. These pairs are notably among the most accessible, as a precise analogue of the Wold decomposition theorem is applicable to them: Let $(V_1, V_2)$ be a doubly commuting pair on $\clh$. Then there exist four closed subspaces $\{\clh_{ij}\}_{i,j = u,s}$ of $\clh$ (some of the spaces might potentially be zero) such that
\begin{equation}\label{eqn Wold DC}
\clh = \clh_{uu} \oplus \clh_{us} \oplus \clh_{su} \oplus \clh_{ss},
\end{equation}
where $\clh_{ij}$ reduces $V_k$ for all $i,j=u,s$, and $k=1,2$, and $V_1|_{\clh_{pq}}$ is a shift if $p=s$ and unitary if $p=u$, and $V_2|_{\clh_{pq}}$ is a shift if $q=s$ and unitary if $q=u$. This result is due to S{\l}oci\'{n}ski \cite{Slo} (also, see \cite{JS} for a more recent account). We will refer to this as the \textit{Wold decomposition for doubly commuting pairs}. Evidently, similar to the case of single isometries, the structure of doubly commuting pairs is explicit and simple (compare \eqref{eqn Wold DC} with \eqref{eqn WD 1}).

The objective of this paper is to examine the structure of the next-best isometric pairs that are more nontrivial in nature. Particularly, we obtain explicit representations as well as a complete set of unitary invariants (numerical in nature) of isometric pairs $(V_1, V_2)$ for which
\[
[V_1^*, V_2] = \text{compact and normal}.
\]
At this point it is worth mentioning that even the characterization of submodules 
$\mathcal{M}$
of the Hardy space over the bidisc $H^2(\mathbb{D}^2)$, for which the cross-commutator $\big[(M_z\vert_{\mathcal{M}})^*, M_w\vert_{\mathcal{M}}\big]$ 
has rank one, is still unknown and remains an open problem-let alone for cases involving finite rank or general compact cross-commutators (see Problem $9$ in \cite[Page 246]{Yang-S}).
In the process of our study, we acquire precise surgery of the delicate structure of isometric pairs as well as many results of independent interest. The outcomes of this paper will not only provide a complete understanding of the above class but also highlight the intricate nature of isometric pairs. 

An immediate simplification of the aforementioned class of isometric pairs is achieved by assuming, without any loss of generality, that the products of these pairs are shift operators (see the end of Section \ref{sec preliminaries} for justification). We note that Berger, Coburn, and Lebow \cite{BCL1978, BCL19785} also looked into such pairs in the context of Fredholm theory and $C^*$-algebras. We formalise the pairs examined by them for future reference:

\begin{definition}\label{def BCL pairs}
A BCL pair is an isometric pair $(V_1, V_2)$ such that the product $V_1V_2$ is a shift operator.
\end{definition}

After this reduction, the following subcategory of isometric pairs emerges as especially compelling in terms of representations and complete unitary invariants.

\begin{definition}\label{def comp norm}
A compact normal pair is an isometric pair $(V_1, V_2)$ such that
\begin{enumerate}
\item $(V_1, V_2)$ is a BCL pair, and
\item $[V_2^*, V_1] = \text{compact and normal}$.
\end{enumerate}
\end{definition}

In the above definition, we consider $[V_2^*, V_1]$ instead of $[V_1^*, V_2]$. Without question, this is an unimportant alteration. Our revised objective now is to represent and then compute a complete set of unitary invariants of compact normal pairs. To that end, we show how compact normal pairs are made up of four distinguished building blocks. Each of the four building blocks is non-unitarily equivalent to the others. However, there is a common hallmark; they bear some connection to the Hardy space over the bidisc. Denote by $\D = \{z \in \mathbb{C}: |z| < 1\}$ the open unit disc in $\mathbb{C}$. The \textit{Hardy space} $H^2(\D^2)$ over the bidisc $\D^2$ is defined by
\[
H^2(\D^2) = \overline{\mathbb{C}[z,w]}^{L^2(\T^2)},
\]
where $\T^2 = \partial \D^2$ is the distinguished boundary of $\D^2$. In view of radial limits, square-summable analytic functions on $\D^2$ can be identified with functions in $H^2(\D^2)$. Moreover, the pair of multiplication operators $(M_z, M_w)$ by the coordinate functions $z$ and $w$, respectively, on $H^2(\D^2)$ satisfy the following key properties:

\begin{enumerate}
\item $(M_z, M_w)$ is a BCL pair.
\item $[M_w^*, M_z] = 0$.
\item Denote by $P_{\mathbb{C}}$ the orthogonal projection onto the space of constant functions of $H^2(\D^2)$. Then
\[
I - M_z M_z^* - M_w M_w^* + M_z M_w M_z^* M_w^* = P_{\mathbb{C}}.
\]
\end{enumerate}

The first and second properties of $(M_z, M_w)$ above already appeared in the definition of compact normal pairs (see Definition \ref{def comp norm}). The final identity is also an inherent characteristic of $(M_z, M_w)$ and serves as the driving force behind the concept of defect operators \cite{Guo}:

\begin{definition}\label{def defect operator}
The defect operator of an isometric pair $(V_1, V_2)$ is the operator $C(V_1, V_2)$ defined by
\[
C(V_1, V_2) = I - V_1 V_1^* - V_2 V_2^* + V_1 V_2 V_1^* V_2^*.
\]
\end{definition}

We are now ready to elucidate the main results of this paper. Our first main result is rather general and related to compact self-adjoint operators that are the differences of two projections. This is a subject in its own right \cite{DJS2016}. Within this context, in Theorem \ref{HiNewThm}, we prove the following result, which is a substantial generalization of the central result of the paper by He et. al. \cite[Theorem 4.3]{HQY2015}. Given a bounded linear operator $T$ on some Hilbert space, we denote by $\sigma(T)$ the spectrum of $T$.

\begin{theorem}
Let $A$ be a compact self-adjoint contraction on a Hilbert space $\clh$. Suppose $A$ is the difference of two projections. If $\lambda \in \sigma(A)\setminus\{0, \pm 1\}$, then
\[
-\lambda \in \sigma(A),
\]
and
\[
\dim \ker(A - \lambda I_{\clh}) = \dim \ker(A + \lambda I_{\clh}).
\]	
\end{theorem}

Now we turn to a rank formula for isometric pairs. Let $(V_1, V_2)$ be an isometric pair. We note that, as a difference of two projections, $C(V_1, V_2)$ is a self-adjoint contraction (see \eqref{formula}). Let $\clh_0$ be \textit{generic part} of $C(V_1, V_2)$ (in the sense of Halmos \cite{Hal}; see Section \ref{sec rank formula} for more details). Then $\clh_0$ is known to reduce $C(V_1, V_2)$. Consider the spectral decomposition of $C(V_1, V_2)|_{\clh_0}$ as
\[
C(V_1, V_2)|_{\clh_0} = \int_{\sigma(C(V_1, V_2)|_{\clh_0})} \lambda \, dE_\lambda,
\]
where $E$ denotes the spectral measure of $C(V_1, V_2)$. The \textit{positive generic part} of $(V_1,V_2)$ is the closed subspace $\clk_+$ defined by (see Definition \ref{def + part} for more details)
\[
\clk_+ = E[0,1]\clh_0.
\]
Also, we define the eigenspaces $E_{\pm 1}$ by
\[
E_{\pm 1} = \ker (C(V_1, V_2) \mp I_{\clh}).
\]
Now we are ready to state the rank theorem (see Theorem \ref{index}).

\begin{theorem}\label{thm intro rank}
Let $(V_1, V_2)$ be an isometric pair. Then
\[
\text{rank} C(V_1, V_2) = \text{rank}[V_2^*, V_1] + \dim E_1 + \dim \clk_+.
\]
If, in addition, $\dim E_{-1} < \infty$, then
\[
\text{rank} C(V_1, V_2) = 2 \text{rank}[V_2^*, V_1] + \dim E_1 - \dim E_{-1}.
\]
\end{theorem}

The above rank formula is at the centre of the paper and will be one of the most effective tools for describing compact normal pairs. In fact, the above theorem will be mostly used for isometric pairs with compact defect operators. And note that the condition
\[
\dim E_{-1} < \infty,
\]
is automatically satisfied for isometric pairs with compact defects (see Corollary \ref{cor rank formula}). Given that the rank formula applies to any isometric pair, we believe that the above result is independently relevant.

Now we turn to representations of compact normal pairs. Our strategy is to split a compact normal pair into smaller pieces of distinguished building blocks. Unquestionably, and at the very least, a distinguished building block must possess the property of irreducibility:
 
\begin{definition}\label{def irred}
Let $(T_1, T_2)$ be a pair of bounded linear operators on $\clh$, and let $\cls \subseteq \clh$ be a closed subspace. We say that $\cls$ reduces $(T_1, T_2)$ (or $\cls$ is $(T_1, T_2)$-reducing) if
\[
T_i\cls, T_i^*\cls \subseteq \cls \qquad (i=1,2).
\]
We say that $(T_1, T_2)$ is irreducible if there is no non-trivial (that is, nonzero and proper) $(T_1, T_2)$-reducing subspace of $\mathcal{H}$.
\end{definition}

Our distinguished building blocks are irreducible (except for shift-unitary pairs; see Definition \ref{def doubly comm} below) and also compact normal pairs. They exhibit a correlation with property (3) from the list that came before Definition \ref{def defect operator}, which concerns the finite rank of the defect operator. This correlation also serves as a motivation to introduce the following class of isometric pairs:

\begin{definition}\label{def n finite}
Let $n \in \mathbb{N}$. An $n$-finite pair is an isometric pair $(V_1, V_2)$ acting on some Hilbert space that meets the following conditions:
\begin{enumerate}
\item $(V_1, V_2)$ is a compact normal pair, and
\item $\text{rank} C(V_1, V_2) = n$.
\end{enumerate}
\end{definition} 

We remark that the above definition and Definition \ref{def BCL pairs} also have some bearing on some classical theories as well as relatively contemporary results of independent interest. We will explain these connections at the end of this section (see the discussion following Theorem \ref{thm intro unit inv}).

We prove that irreducible $n$-finite pairs are the fundamental building blocks of compact normal pairs. Even more strongly, the following is a fact (see Corollary \ref{cor irred n finite}):

\begin{theorem}\label{thm intro 3 pairs}
An irreducible $n$-finite pair is either $1, 2$, or $3$-finite.
\end{theorem}

Subsequently, irreducible $n$-finite pairs, $n = 1, 2, 3$, serve as the fundamental building blocks for compact normal pairs. This observation, along with numerous others, is a result of the robust rank formula that is found in Theorem \ref{thm intro rank}. There will be, however, one more building block (not necessarily irreducible) made up of a simpler class of doubly commuting pairs. To explain this, we recall from \eqref{eqn Wold DC} that there are four summands in the Wold decomposition of doubly commuting pairs. In the present situation, there will be only two summands. More formally:
 

\begin{definition}\label{def doubly comm}
An isometric pair $(V_1, V_2)$ on $\clh$ is said to be shift-unitary if $(V_1, V_2)$ is doubly commuting and the Wold decomposition \eqref{eqn Wold DC} of the pair reduces to
\[
\clh = \clh_{us} \oplus \clh_{su}.
\]
\end{definition}

In Remark \ref{rem shift unit}, we argued that the shift-unitary pairs are indeed simple. There, we will also observe that unitary operators fairly parameterize shift-unitary pairs.

Among $n$-finite pairs, invariant subspaces of $H^2(\D^2)$ will also play an important role. A closed subspace $\cls$ of $H^2(\D^2)$ is an \textit{invariant subspace} if
\begin{equation}\label{eqn: inv sub Hardy}
M_z \cls, M_w \cls \subseteq \cls.
\end{equation}
Before delving into the representations of compact normal pairs, we note that the equality in the category of isometric pairs that we observe will be unitary equivalent. Two isometric pairs $(S_1, S_2)$ on $\clh_1$ and $(T_1, T_2)$ on $\clh_2$  are said to be \textit{unitarily equivalent} if there exists a unitary operator $U: \clh_1 \rightarrow \clh_2$ such that $U S_i = T_i U$ for all $i=1,2$. We often denote this by
\begin{equation}\label{eqn unit equiv}
(S_1, S_2) \cong (T_1, T_2).
\end{equation}
The following result yields concrete representations of compact normal pairs (see Theorem \ref{main}):

\begin{theorem}\label{thm intro}
Let $(V_1, V_2)$ be a compact normal pair on $\clh$. Define
\[
k: = \dim\Big[\ker \Big(C(V_1, V_2) - I_{\clh}\Big)\Big] \in [0, \infty].
\]
Then there exist closed $(V_1, V_2)$-reducing subspaces $\{\clh_i\}_{i=0}^k$ of $\clh$ such that
\[
\clh = \bigoplus_{i=0}^k \clh_i.
\]
If we define
\[
(V_{1,j}, V_{2,j}) = (V_1|_{\clh_j}, V_2|_{\clh_j}) \qquad (j=0, 1, \ldots, k),
\]
then $(V_{1,i}, V_{2,i})$ on $\clh_i$ is irreducible for all $i=1, \ldots, k$. Moreover, we have the following:
\begin{enumerate}
\item $(V_{1,0}, V_{2,0})$ on $\clh_0$ is a shift-unitary type.
\item For each $i=1, \ldots, k$, the pair $(V_{1,i}, V_{2,i})$ on $\clh_i$ is unitarily equivalent to one of the following three pairs:
\begin{enumerate}
\item $(M_z, M_w)$ on $H^2(\D^2)$.
\item $(M_z, {\alpha} M_z)$ on $H^2(\D)$ for some unimodular constant $\alpha$.
\item $(\gamma M_z|_{\cls_{\lambda}}, M_w|_{\cls_{\lambda}})$ on $\cls_{\lambda}$, where
\begin{equation}\label{eqn intro subm}
\cls_{\lambda} = \vp \Big(H^2(\mathbb{D}^2) \bigoplus \Big(\bigoplus_{j = 0}^\infty z^j \text{span} \Big\{ \frac{\bar{w}}{1 - \lambda z \bar{w}}\Big\} \Big) \Big),
\end{equation}
for some $\lambda \in (0,1)$, unimodular constant $\gamma$, and some inner function $\vp \in H^\infty(\D^2)$.
\end{enumerate}
\end{enumerate}
\end{theorem}

The pairs in (a), (b), and (c) in the above theorem represent irreducible $n$-finite pairs, for $n = 1, 2$, and $3$, respectively. First, we comment on irreducible $3$-finite pairs. The invariant subspace $\cls_{\lambda}$ mentioned in part (c) above is attributed to Izuchi and Izuchi \cite{II2006}. This particular invariant subspace leans more towards the existential than the construction. In this case, the inner function $\varphi$, which solely depends on $\lambda$, is derived from Rudin's construction \cite[Theorem 5.4.5]{Rudin book}. We prove that pairs of type (c) above, that is, the Izuchi-Izuchi-Ohno examples, are all irreducible $3$-finite pairs (see Theorem \ref{alphalambda} as well as Theorem \ref{irrfinal}):

\begin{theorem}
Let $(V_1, V_2)$ be an irreducible $3$-finite pair on a Hilbert space $\mathcal{H}$. Then the following hold:
\begin{enumerate}
\item There exists $\lambda \in (0, 1)$ such that
\[
\sigma(C(V_1, V_2)) \cap (0,1)  = \{\lambda\}.
\]
\item There exists a unimodular constant $\gamma$ such that
\[
\sigma([V_2^*, V_1]) \setminus \{0\} = \{\gamma \lambda\}.
\]
\item There exists an inner function $\vp \in H^\infty(\D^2)$ (depending on $\lambda$) such that
\[
(V_1, V_2) \cong (\gamma M_z|_{\cls_{\lambda}}, M_w|_{\cls_{\lambda}}),
\]
where $\cls_{\lambda}$ is the invariant subspace of $H^2(\mathbb{D}^2)$ as in \eqref{eqn intro subm}.
\end{enumerate}
\end{theorem}

In this case, one can prove that (see Proposition \ref{dime1})
\[
\text{dim} E_1(C(V_1, V_2)) = 1,
\]
and
\[
\text{dim} E_{-1}(C(V_1, V_2)) = 0,
\]
and hence, Theorem \ref{thm intro rank} implies that
\[
\text{rank}[V_2^*,V_1] = 1.
\]
Moreover, part (2) of the above theorem says that
\[
\beta:= \gamma \lambda,
\]
is the unique nonzero eigenvalue of the cross-commutator $[V_2^*, V_1]$. Clearly
\[
0<|\beta| < 1.
\]
We prove that this number is a complete unitary invariant (see Theorem \ref{charac}): Let $(V_1, V_2)$ on $\clh$ and $(\tilde{V}_1, \tilde{V}_2)$ on $\tilde{\clh}$ be irreducible $3$-finite pairs. Then
\[
(V_1, V_2) \cong (\tilde{V}_1, \tilde{V}_2),
\]
if and only if
\[
\beta = \tilde \beta,
\]
where $\beta$ and $\tilde \beta$ are the unique nonzero eigenvalues of $[V_2^*, V_1]$ and $[\tilde V_2^*, \tilde V_1]$, respectively.

Next, we turn to irreducible $2$-finite pairs. As in the $3$-finite case, here also we have precise spectral synthesis and a complete representation (see Theorem \ref{rank2}):
 

\begin{theorem}
Let $(V_1, V_2)$ be an irreducible $2$-finite pair.  Then the following hold:
\begin{enumerate}
\item $\{\pm 1\}$ are the only nonzero eigenvalues of $C(V_1, V_2)$.
\item $\text{rank}[V_2^*,V_1] = 1$.
\item There exists a unimodular constant $\alpha$ such that
\[
\sigma([V_2^*,V_1])\setminus \{0\} = \{\alpha\}.
\]
\item $(V_1, V_2) \cong (M_z, \overline{\alpha}M_z)$.
\end{enumerate}
Conversely, if $\alpha$ is unimodular constant, then $(M_z, \overline{\alpha}M_z)$ on $H^2(\mathbb{D})$ is an irreducible $2$-finite pair with $\{\pm 1\}$ as the only nonzero eigenvalues of $C(M_z, \overline{\alpha}M_z)$.
\end{theorem}

As in the case of $3$-finite pairs, the nonzero eigenvalue $\alpha$ in part (3) above is also a complete unitary invariant. More specifically, if $(\tilde{V_1}, \tilde{V_2})$ on $\tilde{\clh}$ is another irreducible $2$-finite pair, then
\[
(V_1, V_2) \cong (\tilde{V_1}, \tilde{V_2}),
\]
if and only if
\[
\alpha = \tilde{\alpha},
\]
where $\tilde{\alpha}$ is the unique nonzero eigenvalue of $[\tilde{V_2}^*,\tilde{V_1}]$ (see Theorem \ref{thm 2 finite unit inv}).

Finally, we discuss irreducible $1$-finite pairs. This class, in contrast to irreducible $3$ and $2$-finite pairs, is relatively simple and precisely one in nature (see Theorem \ref{thm 1 finite}):

\begin{theorem}
Let $(V_1, V_2)$ be an irreducible isometric pair. Then $(V_1, V_2)$ is $1$-finite if and only if
\[
(V_1, V_2) \cong (M_z, M_w) \text{ on } H^2(\D^2).
\]
\end{theorem}

In this case, it is trivial that
\[
\text{rank}[V_2^*,V_1] = 0,
\]
because $(V_1, V_2)$ in particular is a doubly commuting pair.

As pointed out in Theorem \ref{thm intro 3 pairs}, this exhausts the inventory of irreducible $n$-finite pairs. The representations of all irreducible $n$-finite pairs also yield an explicit set of unitary invariants for compact normal pairs: Let $(V_1, V_2)$ on $\clh$ be a compact normal pair. For simplicity, assume that
\[
k := \dim E_1(C(V_1, V_2)) >0.
\]
We follow the outcome of Theorem \ref{thm intro} and call the pair $(V_{1,0}, V_{2,0})$ on $\clh_0$ obtained there the \textit{shift-unitary part} of $(V_1, V_2)$. We show that (see, for instance, Section \ref{sec unitary inv}) for each $i = 1, \ldots, k$, the cross-commutator $[V_{2,i}^*, V_{1,i}]$ is normal and has rank $0$ or $1$. We define the \textit{fundamental sequence} corresponding to the pair $(V_1, V_2)$ to be the sequence of scalars $\{\alpha_i\}_{i =1}^k$ constructed as follows:
\[
\alpha_i :=
\begin{cases}
0 & \mbox{if } \text{ rank}[V_{2, i}^*, V_{1, i}] = 0 \\
\sigma([V_{2, i}^*, V_{1, i}]) \setminus \{0\} & \mbox{if} \text{ rank} [V_{2, i}^*, V_{1, i}] = 1.
\end{cases}
\]
In Theorem \ref{them complete unit inv}, we prove that the above sequence, along with the shift-unitary part (which is simple; see Remark \ref{rem shift unit}), is essentially a complete set of unitary invariants. In fact, we prove much more (see Theorem \ref{them complete unit inv} for more details and the complete version):

\begin{theorem}\label{thm intro unit inv}
Let $(V_1, V_2)$ be a compact normal pair on $\clh$ with $(V_{1,0}, V_{2,0})$ on $\clh_0$ as the shift-unitary part. We have the following:

(i) $(V_1|_{\clh_0^\perp}, V_2|_{\clh_0^\perp})  \cong M_1 \oplus M_2 \oplus M_3$, where $M_i$ is a direct sum of irreducible $i$-finite pairs, $i=1,2,3$.

(ii) Let $(\tilde{V}_1, \tilde{V}_2)$ on $\tilde{\clh}$ be another compact normal pair with the shift-unitary part $(\tilde V_{1,0}, \tilde V_{2,0})$ on $\tilde \clh_0$. Suppose $\{\tilde \alpha_i\}_{i =1}^{\tilde{k}}$ is the associated fundamental sequence with
\[
\tilde{k} =  \dim E_1(C(\tilde V_1, \tilde V_2)) >0.
\]
Then the following are equivalent:
\begin{enumerate}
\item $(V_1, V_2) \cong (\tilde{V_1}, \tilde{V_2})$.
\item $(V_{1,0}, V_{2,0}) \cong (\tilde V_{1,0}, \tilde V_{2,0})$, and $[V_2^*, V_1]|_{{\clh_0}^\perp} \cong [\tilde{V}_2^*, \tilde{V}_1]|_{{\tilde{\clh}_0^\perp}}$.
\item $(V_{1,0}, V_{2,0}) \cong (\tilde V_{1,0}, \tilde V_{2,0})$, $k = \tilde{k}$, and there exists a permutation $\sigma$ of $\{1, 2, \cdots, k\}$ such that
\[
\alpha_i = \tilde \alpha_{\sigma(i)} \qquad (i = 1, 2, \cdots, k).
\]
\end{enumerate}
\end{theorem}

Let us comment on part (3) above. The verification of the unitary equivalence of the shift-unitary part is simple to do (cf. Remark \ref{rem shift unit}). Hence, the determination of the numerical invariant, specifically the fundamental sequence, constitutes the most intricate aspect of the complete unitary invariant.

Now we provide an analysis of the historical context and relevant connections that also underpin the assumptions of $n$-finite pairs and compact normal pairs. As pointed out prior to Definition \ref{def BCL pairs}, BCL pairs were studied by Berger, Coburn, and Lebow \cite{BCL1978} in the context of $C^*$-algebras generated by isometric pairs. One of the keys to their approach was identifying BCL pairs with triples
\[
(\cle, U, P),
\]
where $\cle$ is a Hilbert space, $U$ is a unitary, and $P$ is a projection on $\cle$. The work pointed out the difficulties of the structure of isometric pairs while answering a number of questions along the lines of $C^*$-algebras. While unitary and projection operators are individually among the simplest operators and lack direct interdependence, their interplay gives rise to considerable generality. As such, the structural challenges inherent in isometric pairs persist, though in a transformed guise. Classifying isometric pairs and figuring out computable invariants can still be done with the Berger, Coburn, and Lebow model, but it does so by delving considerably deeper into the realm of linear operator theory and function theory. In this paper, we specifically do so.

Condition (2) in Definition \ref{def comp norm} bears some historical resonance. Let $L_z$ and $L_w$ denote (following Laurent operators) the multiplication by the coordinate functions $z$ and $w$ on $L^2(\T^2)$, respectively. As in \eqref{eqn: inv sub Hardy}, an \textit{invariant subspace} of $L^2(\mathbb{T}^2)$ is a closed subspace $\mathcal{M} \subseteq L^2(\mathbb{T}^2)$ that is invariant under both $L_z$ and $L_w$. Let $\clm$ be an invariant subspace of $L^2(\T^2)$. Clearly, $(L_z|_{\clm}, L_w|_{\clm})$ is an isometric pair on $\clm$. Nakazi speculated in \cite{Nakazi} that
\[
[(L_w|_{\clm})^*, L_z|_{\clm}] = [(L_w|_{\clm})^*, L_z|_{\clm}]^*,
\]
if and only if
\[
[(L_w|_{\clm})^*, L_z|_{\clm}] = 0.
\]
However, in \cite{IO94}, Izuchi and Ohno  provided concrete examples (and hence repudiate Nakazi's conjecture) of invariant subspaces $\clm$ of $L^2(\T^2)$ for which
\[
[(L_w|_{\clm})^*, L_z|_{\clm}] = [(L_w|_{\clm})^*, L_z|_{\clm}]^* \neq 0.
\]
In fact, they fully characterized invariant subspaces of $L^2(\T^2)$ with these characteristics in the same paper. Then Nakazi restricted his question to the analytic part of $L^2(\T^2)$. This time, $\clm$ is an invariant subspace of $H^2(\D^2)$ (see Definition \ref{eqn: inv sub Hardy} for the notion of invariant subspaces). Set
\begin{equation}\label{eqn rest of Hardy inv}
R_z = M_z|_{\clm} \text{ and } R_w = M_w|_{\clm}.
\end{equation}
Then $(R_z, R_w)$ is a BCL pair on $\clm$ (see the closing remark in Section \ref{sec rank formula}). Nakazi refined his question in terms of the existence of invariant subspaces $\clm$ of $H^2(\D^2)$ for which
\[
[R_w^*,R_z] = [R_w^*,R_z]^* \neq 0.
\]
Eventually, K. J. Izuchi and K. H. Izuchi showed intricate examples of such invariant subspaces in \cite{II2006}. Their construction relied heavily on Rudin's description of a specific class of inner functions \cite[Theorem 5.4.5]{Rudin book}. Curiously, all of the cross-commutators provided as examples by Izuchi-Ohno and Izuchi-Izuchi are rank one. In our case, condition (2) is more broad; it needs the cross-commutator to be compact normal rather than rank-one self-adjoint. On the other hand, as demonstrated in Theorem \ref{thm intro}, the representations of irreducible $3$-finite pairs and the construction of examples by Izuchi and Izuchi serve as one of the fundamental building blocks in the representation of compact normal pairs. For more on not-so-standard invariant subspaces of $H^2(\D^2)$, we refer the reader to \cite{ACD, AC, R-jfa}.

We now make some remarks about defect operators (see Definition \ref{def defect operator}). In the context of isometric dilations for pairs of commuting contractions, defect operators were studied decades ago. However, in the context of Beurling-type properties of invariant subspaces of $H^2(\D^2)$ this was analyzed  more closely in \cite{Guo}. Defect operators have been a useful tool in the theory of invariant subspaces of $H^2(\D^2)$. However, see \cite{ACD} for some deviations.

The remaining part of the paper is organized as follows: Section \ref{sec preliminaries} collects some well-known results and fixes notations for future uses of the paper. In this section, we briefly recall the analytic models of BCL pairs. The model that is presented here \cite{MSS}, which will also be extensively used, differs slightly from the original model of Berger, Coburn, and Lebow. The present model is more explicit and, hence, more useful. Here we also prove a result concerning the eigenspaces of compact self-adjoint contractions that are the differences of two projections. This yields a generalization of one of the main results of \cite{HQY2015}. Section \ref{sec rank formula} is the heart of this paper, in which we establish a very useful rank formula for isometric pairs.

As previously mentioned, there are only three irreducible $n$-finite pairs: $n = 1, 2$, and $3$. Out of all the pairs, irreducible $3$-finite pairs are the most complex. Section \ref{sec properties of 3 finite} presents an overview of the characteristics of irreducible $3$-finite pairs. This encompasses an extensive spectral synthesis of irreducible $3$-finite pairs. In Section \ref{sec class of 3 finite}, we proceed with irreducible $3$-finite pairs and explicitly represent them by using the results from the previous section, coupled with additional observations. We also point out that a certain scalar, namely, the unique nonzero eigenvalue of the pairs' cross-commutator, is the complete unitary invariant of $3$-finite pairs. In Section \ref{sec 2 finite}, we represent and classify irreducible $2$-finite pairs. The representation of irreducible $1$-finite pairs is the easiest, and that has been completely explored in the short section, namely Section \ref{sec 1 finite}.

Section \ref{sec class n finite} provides an in-depth description of compact normal pairs. This section compiles all of the machinery developed in previous sections and creates the building blocks required for constructing compact normal pairs. The representation also requires some structural results that have also been developed in this section. Section \ref{sec unitary inv} uses the representations of compact normal pairs obtained in the preceding section to present a complete set of unitary invariants. This section, and hence the paper, concludes with some natural questions for future investigation.

\newsection{Preliminaries}\label{sec preliminaries}

This section serves to establish some notation and refresh the reader's memory of some well-known results that will be applied throughout the remainder of the paper. Along the way, in Theorem \ref{HiNewThm}, we prove a result concerning the eigenspace of certain compact self-adjoint operators that generalizes a result earlier proved in \cite[Theorem 4.3]{HQY2015}. We begin with the analytic representations of shift operators. The \textit{defect operator} of an isometry $T$ acting on some Hilbert space is defined by the operator $I-TT^*$. Observe that the defect operator of $T$ is a measure of $T$ not being unitary. Let $T$ be an isometry on $\clh$. The \textit{wandering subspace} of $T$ is defined by
\[
\mathcal{E} :=\ker T^*.
\]
Clearly, $\mathcal{E} = \text{ran} (I - TT^*)$. It is easy to see that
\[
T^n \cle \perp T^m \cle,
\]
for all $m \neq n$, $m, n \in \Z_+$. If, in addition, $T$ is a shift (that is, $T^{*m} \rightarrow 0$ in SOT; cf. \eqref{eqn def shift}), then we have the orthogonal decomposition of $\clh$ as
\begin{equation}\label{eqn: H Wold decomp}
\mathcal{H} = \bigoplus_{n = 0}^\infty T^n\mathcal{E}.
\end{equation}
This amounts to saying that $T$ is unitarily equivalent to $M_z$ on $H^2_{\cle}(\D)$, where $M_z$ is the operator of multiplication by the coordinate function $z$, and $H^2_{\cle}(\D)$ denotes the $\cle$-valued Hardy space over $\D$. We will often identify $H^2_{\cle}(\D)$ with the Hilbert space tensor product $H^2(\D) \otimes \cle$. Bearing this in mind, we will also often identify $M_z$ on $H^2_\cle(\D)$ with $M_z \otimes I_{\cle}$ on $H^2(\D) \otimes \cle$.

Now we turn to isometric pairs. Given an isometric pair $(V_1, V_2)$ on $\clh$, we will use the following notational convention throughout this paper:
\[
V = V_1 V_2,
\]
and
\[
\clw_i = \ker V_i^*,
\]
for all $i=1,2$. We also set
\[
\clw = \ker V^*.
\]
Therefore, $\clw$ and $\clw_i$ are wandering subspaces corresponding to the isometries $V$ and $V_i$, $i=1,2$, respectively. Recall that the defect operator is defined by (see Definition \ref{def defect operator})
\[
C(V_1, V_2) = I - V_1 V_1^* - V_2 V_2^* + V_1 V_2 V_1^* V_2^*.
\]
Since the right side of the above equals
\[
(I - V_1 V_1^*) - V_2 (I - V_1 V_1^*) V_2^* = (I - V_2 V_2^*) - V_1 (I - V_2 V_2^*) V_1^*,
\]
it follows that
\begin{equation}\label{formula}
C(V_1, V_2) = P_{\clw_1} - P_{V_2 \clw_1} = P_{\clw_2} - P_{V_1 \clw_2}.
\end{equation}
Given a Hilbert space $\clk$ and a closed subspace $\cls$, denote by $P_\cls$ the orthogonal projection onto $\cls$. It is useful to note that $C(V_1, V_2)$ is a self-adjoint operator, that is
\[
C(V_1, V_2)^* = C(V_1, V_2).
\]
Furthermore, as a difference of projections, it is readily apparent that $C(V_1, V_2)$ is a contraction.

We will now discuss the analytic structure of BCL pairs (see Definition \ref{def BCL pairs}). We start with a definition:

\begin{definition}
A BCL triple is an ordered triple $(\cle, U, P)$ consisting of a Hilbert space $\cle$, and a unitary $U$ and a projection $P$ acting on $\cle$.
\end{definition}

BCL pairs and BCL triples are interchangeable; the explanation is as follows: Let $(\cle, U, P)$ be a BCL triple. Consider Toeplitz operators $M_{\Phi_1}$ and $M_{\Phi_2}$ with polynomial analytic symbols
\[
\Phi_1(z) = (P^\perp + z P)U^*, \text{ and } \Phi_2(z) = U(P + zP^\perp) \qquad (z \in \D),
\]
where $P^{\perp} :=I - P$. It is easy to see that $(M_{\Phi_1}, M_{\Phi_2})$ is a BCL pair on $H^2_{\cle}(\D)$. Since, up to unitary equivalence, a shift operator is the multiplication operator $M_z$ on some vector-valued Hardy space (see \eqref{eqn: H Wold decomp} above), an application of the Beurling-Lax-Halmos theorem yields the converse of the above construction. That is, if $(V_1, V_2)$ is a BCL pair on $\clh$, then there exist a BCL triple $(\mathcal{E}, U, P)$ and a unitary $\eta: \clh \rightarrow H^2_{\cle}(\D)$ such that $\eta V_i = M_{\Phi_i} \eta$ for all $i=1,2$. The notation for this is (see \eqref{eqn unit equiv})
\[
(V_1, V_2) \cong (M_{\Phi_1}, M_{\Phi_2}).
\]
This provides an analytic representation of the class of BCL pairs (see \cite{BCL1978} for complete details).

\begin{theorem}\label{thm BCL classic}
Up to joint unitary equivalence, BCL pairs are of the form $(M_{\Phi_1}, M_{\Phi_2})$ on $H^2_{\cle}(\D)$ for BCL triples $(\mathcal{E}, U, P)$.
\end{theorem}

Given a BCL pair, following \cite{MSS}, one can construct the corresponding BCL triple more explicitly. The analytic model of BCL pairs presented below is an explicit reformulation of the original result of Berger, Coburn, and Lebow, which will play an influential role in our analysis.

\begin{theorem}\cite[Lemma 3.1 and Theorem 3.3]{MSS}\label{bcl}
Let $(V_1, V_2)$ be an isometric pair. Then
\[
\clw = \clw_1 \oplus V_1 \clw_2 = \clw_2 \oplus V_2 \clw_1,
\]
and the operator
\[
U = \begin{bmatrix}
V_2|_{\clw_1} & \\
& V_1^*|_{V_1 \clw_2}
\end{bmatrix} : \clw_1 \oplus V_1 \clw_2 \to V_2\clw_1 \oplus \clw_2,
\]
defines a unitary on $\mathcal{W}$. Moreover, if $(V_1, V_2)$ is a BCL pair, then $(\clw, U, P_{\clw_1})$ is the BCL triple corresponding to $(V_1, V_2)$. In particular, $(V_1, V_2, V_1V_2)$ on $\mathcal{H}$ and $(M_{\Phi_1}, M_{\Phi_2}, M_z)$ on $H^2_{\clw}(\mathbb{D})$ are unitarily equivalent, where
\[
\Phi_1(z) = (P_{\clw_1}^\perp + z P_{\clw_1})U^*, \text{ and } \Phi_2(z) = U(P_{\clw_1} + zP_{\clw_1}^\perp) \qquad (z \in \D).
\]
\end{theorem}

Because of this, in what follows, given a BCL pair $(V_1, V_2)$, we will use the representation $(M_{\Phi_1}, M_{\Phi_2})$ of $(V_1, V_2)$, which corresponds to the BCL triple $(\mathcal{W}, U, P_{\mathcal{W}_1})$.

Let $(V_1, V_2) = (M_{\Phi_1}, M_{\Phi_2})$ be a BCL pair on $H^2_{\clw}(\D)$. A simple computation yields
\[
I - M_{\Phi_1} M_{\Phi_1}^* = P_{\mathbb{C}} \otimes P_{\clw_1},
\]
and
\[
I - M_{\Phi_2} M_{\Phi_2}^* = P_{\mathbb{C}} \otimes U P_{\clw_1}^\perp U^*.
\]
With respect to the orthogonal decomposition $H^2_{\clw}(\D) = \clw \oplus zH^2_{\clw}(\D)$, it also follows that
\[
C(V_1, V_2) = C(M_{\Phi_1}, M_{\Phi_2}) = \begin{bmatrix}
P_{\clw_1} - U P_{\clw_1} U^*  & 0
\\
0 & 0
\end{bmatrix}.
\]
Consequently
\[
(\ker C(V_1, V_2))^\perp \subseteq \clw,
\]
and hence, it suffices to study $C(V_1, V_2)$ only on $\clw$.

At this moment, we need to pause and relook at \eqref{formula}, which says that the defect operator of an isometric pair is a difference of two projections. We also pointed out that the defect operator is a contraction and a self-adjoint operator. This viewpoint was employed in \cite[Theorem 4.3]{HQY2015} to examine the eigenspace structure of defect operators of isometric pairs. With this as motivation, we now explore the eigenspace structure of compact self-adjoint contractions that can be represented as a difference of pairs of projections. First, we set up a notation for eigenspace. For each $\mu \in \mathbb{C}$ and bounded linear operator $X$ on some Hilbert space $\clk$, we define the eigenspace
\begin{equation}\label{eqn def eigenspace 1}
E_{\mu}(X):= \ker \big(X - \mu I_{\clk}\big).
\end{equation}
Suppose $A$ is a self-adjoint contraction acting on a Hilbert space $\clh$. Then $\ker A$, $E_1(A)$, and $E_{-1}(A)$ reduce $A$, and hence there exists a closed $A$-reducing subspace $\clh_0 \subseteq \clh$ such that \cite{Hal}
\begin{equation}\label{eqn H as generic}
\clh = \ker A \oplus E_1(A) \oplus E_{-1}(A) \oplus \mathcal{H}_0.
\end{equation}
The part $A_0:= A|_{\clh_0}$ is known as the \textit{generic part} of $A$ (see Halmos \cite{Hal} for more details). Let $E$ denote the spectral measure of $A$. Then
\[
A_0 = \int_{\sigma(A_0)} \lambda dE_\lambda,
\]
is the spectral representation of $A_0$. Define closed $A_0$-reducing subspaces $\clk_+$ and $\clk_-$ by
\[
\clk_+ = E[0, 1]\clh_0,
\]
and
\[
\clk_- = E[-1, 0]\clh_0.
\]
Now, we prove the eigenspace property for $A$ by assuming that it is the difference of two projections. Although this result is a consequence of
\cite[Proposition 2.1]{DJS2016}, we provide a proof here for the sake of completeness and readers
convenience. The spectral theorem of compact self-adjoint operators, certain projection methods from \cite{DJS2016}, and the above-mentioned Halmos constructions provide the foundation of the proof.

\begin{theorem}\label{HiNewThm}
Let $A$ be a compact self-adjoint contraction on a Hilbert space $\clh$. Suppose $A$ is the difference of two projections. If $\lambda \in \sigma(A)\setminus\{0, \pm 1\}$, then
\[
-\lambda \in \sigma(A),
\]
and
\[
\dim E_\lambda(A) = \dim E_{-\lambda}(A).
\]	
\end{theorem}
\begin{proof}
We proceed with the orthogonal decomposition of $\clh$, the spectral representation of $A$, and the notations $\clk_+$ and $\clk_-$ introduced prior to stating this theorem. Define the restriction operators
\[
\begin{cases}
{A_0}_+ = A_0|_{\clk_+}
\\
{A_0}_- = -A_0|_{\clk_-}.
\end{cases}
\]
Observe that ${A_0}_+$ and ${A_0}_-$ are the positive and negative parts of $A_0$, respectively. We have the matrix representation
\[
A_0 = \begin{bmatrix}
{A_0}_+ &  \\
 & -{A_0}_-
\end{bmatrix},
\]
on
\[
\clh_0 = \clk_+ \oplus \clk_-.
\]
Suppose further that $A$ is a difference of two projections. By \cite[Proposition 2.1]{DJS2016} and \cite[Remark 3.1]{DJS2016}, there is a unitary operator $u: \clk_+ \rightarrow \clk_-$ such that
\[
{A_0}_- = u {A_0}_+ u^*.
\]
We also define Hilbert spaces
\[
\clh_+ = \ker A \oplus E_1(A) \oplus E_{-1}(A) \oplus \clk_+ \oplus \clk_+,
\]
and
\[
\clh_- = \ker A \oplus E_1(A) \oplus E_{-1}(A) \oplus \clk_+ \oplus \clk_-.
\]
Then
\[
U := I_{\ker A} \oplus I_{E_1(A)} \oplus I_{E_{-1}(A)} \oplus I_{\clk_+} \oplus u,
\]
defines a unitary operator $U: \clh_+ \rightarrow \clh_-$. Therefore, the operator
\[
\tilde{A} := U^* A U:\clh_+ \rightarrow \clh_+,
\]
admits the following block-diagonal operator matrix representation
\begin{equation}\label{HiNew00}
\tilde{A} =
\begin{bmatrix}
0 & & & \\
& I & & \\
& & -I & \\
& & & {A_0}_+ & \\
& & & & -{A_0}_+\\
\end{bmatrix}.
\end{equation}
Since $A$ is compact, it then follows from the spectral theorem for compact self-adjoint operators that the spectrum
\[
\sigma(A) \subseteq [-1, 1],
\]
is a countable set, and
\[
\sigma(A_0) = \sigma(A) \cap (-1, 1)\setminus\{0\}.
\]
From the definitions of $\clk_+$ and $\clk_-$, it follows that
\[
\clk_+ = \underset{{\lambda \in \sigma(A_0) \cap (0, 1)}}{\bigoplus}E_\lambda(A_0),
\]
and
\[
\clk_- = \underset{{\lambda \in \sigma(A_0) \cap (-1, 0)}}{\bigoplus}E_\lambda(A_0).
\]
Therefore
\[
\clk_+ = \underset{{\lambda \in \sigma(A) \cap (0, 1)}}{\bigoplus}E_\lambda(A),
\]
and
\[
\clk_- = \underset{{\lambda \in \sigma(A) \cap (-1, 0)}}{\bigoplus}E_\lambda(A).
\]
Also, it follows from the unitary equivalence of ${A_0}_+$ on $\clk_+$ and ${A_0}_-$ on $\clk_-$ that
\[
\sigma({A_0}_+) = \sigma({A_0}_-),
\]
and consequently
\[
\begin{split}
\sigma(A_0) & = \sigma(A_0|_{\clk_+}) \cup \sigma(A_0|_{\clk_-})
\\
& = \sigma({A_0}_+) \cup \sigma(-{A_0}_-)
\\
& = \sigma({A_0}_+) \cup -\sigma({A_0}_-)
\\
& = \sigma({A_0}_+) \cup -\sigma({A_0}_+).
\end{split}
\]
We conclude that
\begin{equation}\label{HiNew3}
\lambda \in \sigma(A_0) \text{ if and only if } -\lambda \in \sigma(A_0).
\end{equation}
Moreover, for each $\lambda \in \sigma({A_0}_+)$, the unitary equivalence of ${A_0}_+$ on $\clk_+$ and ${A_0}_-$ on $\clk_-$ yields
\[
\dim E_\lambda({A_0}_+) = \dim E_\lambda({A_0}_-).
\]
Therefore,   for $\lambda \in \sigma({A_0}_+)$, we have
\[
\begin{split}
\dim E_\lambda(A) & = \dim E_\lambda({A_0}_+)
\\
& = \dim E_\lambda({A_0}_-)
\\
& = \dim E_\lambda(-A_0|_{\clk_-})
\\
& = \dim E_{-\lambda}(A_0|_{\clk_-})
\\
& = \dim E_{-\lambda}(A).
\end{split}
\]
This completes the proof of the theorem.
\end{proof}

Note that using \eqref{HiNew3}, we also have that
\begin{equation}\label{eqn K+}
\clk_+ = \underset{{\lambda \in \sigma(A) \cap (0, 1)}}{\oplus}E_\lambda(A)
\end{equation}
and
\begin{equation}\label{eqn K-}
\clk_- = \underset{{\lambda \in \sigma(A) \cap (0, 1)}}{\oplus}E_{-\lambda}(A).
\end{equation}

Theorem \ref{HiNewThm} significantly unifies a result previously established in \cite[Theorem 4.3]{HQY2015} within the framework of isometric pairs. Since we will be needing the particular version of \cite[Theorem 4.3]{HQY2015}, we elaborate on it in full detail. Let $(V_1, V_2)$ be a BCL pair on $\clh$. Assume, in addition, that $C(V_1, V_2)$ is a compact operator. Note that $C(V_1, V_2) $ is a self-adjoint contraction since it is the difference of two projections (see the remarks following \eqref{formula}). Therefore, as observed earlier
\[
\sigma(C(V_1, V_2)) \subseteq [-1, 1].
\]
Recall from \eqref{eqn def eigenspace 1}, for a bounded linear operator $X$ on a Hilbert space $\clk$, the eigenspace corresponding to $\mu \in \mathbb{C}$ is denoted by
\[
E_{\mu}(X):= \ker \big(X - \mu I_{\clk}\big).
\]
In the case of our isometric pair $(V_1, V_2)$, we simplify the notation as
\begin{equation}\label{eqn def eigenspace 2}
E_{\mu}:= E_{\mu}(C(V_1, V_2)) = \ker \big(C(V_1, V_2) - \mu I \big).
\end{equation}
Set
\[
\Lambda = \sigma(C(V_1, V_2)) \cap (0,1).
\]
Note that $\Lambda$ is at most a countable set. Let $\lambda \in (0,1)$. By Theorem \ref{HiNewThm}, it follows that if $\lambda \in \Lambda$, then $-\lambda \in \sigma(C(V_1, V_2))$, and
\[
\dim E_{\lambda} = \dim E_{-\lambda}.
\]
Since $C(V_1, V_2)$ is a compact self-adjoint operator, this implies
\begin{equation}\label{eq-C perp E1}
(\ker C(V_1, V_2))^{\perp} = E_1 \bigoplus_{\lambda \in \Lambda} E_{\lambda} \bigoplus E_{-1} \bigoplus_{\lambda \in \Lambda} E_{-\lambda},
\end{equation}
and $C(V_1, V_2)|_{(\ker {C(V_1, V_2)})^{\perp}}$ is unitarily equivalent to the diagonal block matrix:
\begin{equation}\label{strucdiag}
C(V_1, V_2)|_{(\ker {C(V_1, V_2)})^{\perp}} \cong
\left[ \begin{array}{cccc}
I_{\mathbb{C}^{l_1}} & 0 & 0 & 0 \\
0 & D & 0 & 0\\
0 & 0 & - I_{\mathbb{C}^{l_{-1}}}  & 0 \\
0 & 0 & 0 & -D\\
\end{array}
\right],
\end{equation}
where $l_1 = \mbox{dim} E_1$, $l_{-1} = \mbox{dim} E_{-1}$, $D = \bigoplus_{\lambda} \lambda I_{\mathbb{C}^{k_{\lambda}}}$ , and
\[
k_{\lambda} = \mbox{dim} E_{\lambda} = \mbox{dim} E_{-\lambda}.
\]
Note that $l_1, l_{-1} \in \mathbb{Z}_+$. Combining the results mentioned above yields the following, which recovers \cite[Theorem 4.3]{HQY2015}. This result will be another important tool for what we do in the next sections.

\begin{theorem}\label{structure}
Let $(V_1, V_2)$ be a BCL pair with a compact defect operator. Then for each
\[
\lambda \in \sigma(C(V_1, V_2)) \setminus \{0, \pm 1\},
\]
we have $-\lambda \in \sigma(C(V_1, V_2))$, and
\[
\dim E_{\lambda} = \dim E_{-\lambda}.
\]
Moreover, the nonzero part of the defect operator $C(V_1, V_2)$ is unitarily equivalent to a block diagonal matrix of the form \eqref{strucdiag}.
\end{theorem}

We conclude this section by elucidating the rationale behind the study of BCL pairs among the set of isometric pairs. In fact, the primary obstacle to the characterization problem of isometric pairs is the characterization of BCL pairs. For if $(V_1, V_2)$ is an isometric pair on $\mathcal{H}$, then applying the von Neumann-Wold theorem to
\[
V:= V_1V_2,
\]
one finds unique orthogonal decomposition (see \eqref{eqn WD 1})
\[
\clh = \clh_u \oplus \clh_s,
\]
where $\clh_u$ and $\clh_s$ are closed $V$-reducing subspaces, and $V|_{\clh_u}$ is a unitary, and $V|_{\clh_s}$ is a shift. One can easily show that $\clh_u$ and $\clh_s$ are $(V_1, V_2)$-reducing subspaces \cite[Lemma 6.1]{MSS}. Therefore, $(V_1|_{\mathcal{H}_u}, V_2|_{\mathcal{H}_u})$ is a commuting pair of unitaries and $(V_1|_{\mathcal{H}_s}, V_2|_{\mathcal{H}_s})$ is a BCL pair. As we have a fair understanding of pairs of commuting unitaries (like a definite spectral theorem for tuples of commuting unitaries or even normal operators), it is natural to shift our attention solely to the category of BCL pairs. 

\section{A rank formula}\label{sec rank formula}

The goal of this section is to link the ranks of defect operators and cross-commutators of isometric pairs. This result will be extensively used thereafter. The rank result might be interesting by itself.

First, we again consider the problem of representing self-adjoint contractions, which are the differences of two projections. Recall from the proof of Theorem \ref{HiNewThm} that if $A$ is a self-adjoint contraction on a Hilbert space $\clh$, which is the difference of two projections, then up to unitary equivalence, $\clh$ admits the orthogonal decomposition
\[
\clh = \ker A \oplus \ker(A-I) \oplus \ker (A+I) \oplus \clk \oplus \clk,
\]
for some closed subspace $\clk$ of $\clh$, and with respect to this decomposition of $\clh$, $A$ admits the block-diagonal operator matrix representation (see \eqref{HiNew00})
\[	
A =
\begin{bmatrix}
0 & & & \\
& I & & \\
& & -I & \\
& & & D & \\
& & & & -D\\
\end{bmatrix},
\]
where $D$ is a positive contraction on $\clk$. In other words, the operator $A$ is an example of an operator that can be represented as the difference of pairs of projections. Moreover, the pair of projections can be completely parameterized. More specifically  \cite[Theorem 3.2]{DJS2016}:

\begin{theorem}\label{struc}
With notations as above, the diagonal operator $A$ is a difference of two projections. Moreover, if $A = P - Q$ for some projections $P$ and $Q$, then there exist a projection $R$ defined on $\ker A$ and a unitary $U$ on $\mathcal{K}$ commuting with $D$ on $\mathcal{K}$ such that
\[
P = R \oplus I \oplus 0 \oplus P_U \quad \mbox{and} \quad Q = R \oplus 0 \oplus I \oplus Q_U,
\]
where $P_U$ and $Q_U$ are projections on $\mathcal{K} \oplus \mathcal{K}$ defined by
\[
P_U = \frac{1}{2} \left[ \begin{array}{cc}
I + D & U(I-D^2)^{\frac{1}{2}} \\
U^* (I-D^2)^{\frac{1}{2}} & I - D\\
\end{array}
\right],
\]
and
\[
Q_U = \frac{1}{2} \left[ \begin{array}{cc}
I - D & U(I-D^2)^{\frac{1}{2}} \\
U^* (I-D^2)^{\frac{1}{2}} & I + D\\
\end{array}
\right].
\]
\end{theorem}

The above result is one of the tools that will be utilized for proving the rank formula. We also need to compute the ranks of $P_U$ and $Q_U$ that we do in the following lemma. Part of the proof of the lemma is motivated by \cite[Theorem 3.3]{DSPS}.

\begin{lemma}\label{rank}
In the setting of Theorem \ref{struc}, we have the following identity:
\[
\text{rank } P_U = \text{rank }Q_U = \dim \mathcal{K}.
\]
\end{lemma}
\begin{proof}
For each $x \in \clk$, the representations of $P_U$ and $Q_U$ imply
\[
P_U(x \oplus 0) = \frac{I + D}{2}x \oplus U^* \frac{(I - D^2)^{\frac{1}{2}}}{2}x,
\]
and
\[
Q_U(0 \oplus x) = U \frac{(I - D^2)^{\frac{1}{2}}}{2}x \oplus \frac{I + D}{2}x.
\]
Note that $D$ is a positive contraction, and hence
\[
x = 0,
\]
whenever
\[
P_U(x \oplus 0) = 0,
\]
or
\[
Q_U(0 \oplus x) = 0.
\]
Consequently
\[
P_U|_{\clk \oplus \{0\}}:\clk \oplus \{0\} \to \clk \oplus \clk,
\]
and
\[
Q_U|_{\{0\} \oplus \clk} : \{0\} \oplus \clk \to \clk \oplus \clk,
\]
are injective operators. Therefore, if
\[
\dim \clk = \infty,
\]
we clearly have
\[
\text{rank }P_U = \text{rank }Q_U = \dim \clk (= \infty).
\]
Now assume that
\[
\dim \clk < \infty.
\]
In this case, $D$ as well as $I-D$ are positive and invertible operators. If $x \in \clk$, then, as in the first part of the proof of this lemma, we compute
\[
\begin{split}
P_U(x \oplus 0) & = \frac{I + D}{2}x \oplus U^* \frac{(I - D^2)^{\frac{1}{2}}}{2}x
\\
& = P_U \Big(0 \oplus U^*\sqrt{\frac{I+D}{I-D}}x\Big),
\end{split}
\]
and by duality
\[
\begin{split}
Q_U(0 \oplus x) & = U \frac{(I - D^2)^{\frac{1}{2}}}{2}x \oplus \frac{I + D}{2}x
\\
& = Q_U \Big(U\sqrt{\frac{I+D}{I-D}}x \oplus 0\Big).
\end{split}
\]
So we find
\[
\ran P_U = \{P_U(x \oplus 0) : x \in \clk\},
\]
and
\[
\ran\, Q_U = \{ Q_U(0 \oplus x) : x \in \clk\}.
\]
Moreover, the vectors on the right-hand sides of $P_U(x \oplus 0)$ and $Q_U(0 \oplus x)$ in the above pair of equalities readily imply that $\tau: \ran P_U \rightarrow \ran Q_U$ defined by
\[
\tau (P_U(x \oplus 0)) = Q_U(0 \oplus x) \qquad (x \in \clk),
\]
is a linear isomorphism. In particular
\[
\mbox{rank} P_U = \mbox{rank} Q_U.
\]
Also, the map
\[
\clk \ni x \mapsto P_U(x \oplus 0) \in \text{ran}P_U,
\]
is clearly a linear isomorphism, which yields
\[
\dim \clk = \text{rank}P_U.
\]
Thus, we have proved that $\dim \clk = \text{rank}P_U = \text{rank}Q_U$. This completes the proof of the lemma.	
\end{proof}

Now we return to isometric pairs. Let $(V_1, V_2)$ be an isometric pair on $\clh$. In the upcoming discussion, we will closely adhere to the strategy laid out in the proof of Theorem \ref{HiNewThm}. Additionally, we will use all the notations that were presented at the outset of Section \ref{sec preliminaries} for isometric pairs. As an example, recall that $\clw = \ker (V_1V_2)^*$. Let
\[
\mathcal{N} = \mathcal{W} \ominus \big(\ker C(V_1, V_2)\big)^\perp,
\]
and also set
\[
\clh_0 = \clw \ominus \big(\mathcal{N} \oplus E_1 \oplus E_{-1}\big).
\]
Therefore
\[
\clw = \mathcal{N} \oplus E_1 \oplus E_{-1} \oplus \clh_0.
\]
This decomposition is comparable with \eqref{eqn H as generic}. Therefore, following the discussion preceding Theorem \ref{HiNewThm}, we recognize that $C(V_1, V_2)|_{\clh_0}$ is the generic part of $C(V_1, V_2)$, and then we consider the spectral decomposition of $C(V_1, V_2)|_{\clh_0}$ as
\[
C(V_1, V_2)|_{\clh_0} = \int_{\sigma\big(C(V_1, V_2)|_{\clh_0}\big)}\lambda dE_\lambda.
\]
Similarly, we also set
\[
\clk_+ = E[0,1]\clh_0,
\]
and
\[
\clk_- = E[-1,0]\clh_0.
\]
We take a brief break in order to offer a definition for later usage.

\begin{definition}\label{def + part}
Let $(V_1, V_2)$ be an isometric pair. The positive generic part of $(V_1,V_2)$ is the closed subspace $\clk_+$ defined by
\[
\clk_+ = E[0,1]\clh_0.
\]
\end{definition}

In other words, the generic part of $(V_1,V_2)$ is the closed subspace corresponding to the positive part of the generic part of $C(V_1, V_2)$.

Returning to our setting of isometric pair $(V_1, V_2)$, we therefore have
\begin{equation}\label{eqn rep of W}
\mathcal{W} = \mathcal{N} \oplus E_1 \oplus E_{-1} \oplus \mathcal{K}_+ \oplus \clk_-.
\end{equation}
With respect to this decomposition, we represent $C(V_1, V_2)|_{\clw}$ as
\begin{equation}\label{eqn C(V1 V2)}
C(V_1, V_2)|_{\clw} = \left[ \begin{array}{ccccc}
0_{\mathcal{N}} &  & & & \\
& I_{E_1} & & & \\
& & -I_{E_{-1}}  & &  \\
& &  &  C(V_1, V_2)|_{\clk_+} & \\
& & & & C(V_1, V_2)|_{\clk_-}\\
\end{array}
\right].
\end{equation}
As $C(V_1, V_2)$ is a difference of two projections, as in the proof of Theorem \ref{HiNewThm}, there is a unitary $u: \clk_+ \to \clk_-$ such that
\begin{equation}\label{HiNew6}
u C(V_1, V_2)|_{\clk_+}u^* = - C(V_1, V_2)|_{\clk_-}.
\end{equation}
Define
\begin{equation}\label{eqn tilde E}
\tilde{\mathcal{E}} := \mathcal{N} \oplus E_1 \oplus E_{-1} \oplus \mathcal{K}_+ \oplus \mathcal{K}_+.
\end{equation}
Consequently, we have the unitary operator (recall the representation of $\clw$ in \eqref{eqn rep of W}) \begin{equation}\label{eqn Def of U}
U:= I_{\mathcal{N}} \oplus I_{E_1} \oplus I_{E_{-1}} \oplus I_{\mathcal{K}_+} \oplus u : \tilde{\cle} \longrightarrow \mathcal{W}.
\end{equation}
Set
\[
\tilde{C} := U^* C(V_1, V_2) U : \tilde{\mathcal{E}} \rightarrow \tilde{\mathcal{E}}.
\]
With respect to the decomposition of $\tilde{\cle}$ as in \eqref{eqn tilde E}, we have
\[
\tilde{C} = \left[ \begin{array}{ccccc}
0_{\mathcal{N}} & & &
\\
& I_{E_1} & & &
\\
& & -I_{E_{-1}}  & &
\\
& &  & D &
\\
& &  &  & -D
\\
\end{array}
\right]
\]
where
\[
D = C(V_1, V_2)|_{\clk_+}.
\]
Now, by \eqref{formula}, we know that $C(V_1,V_2)$ can be expressed as a difference of projections:
\[
C(V_1,V_2) = P_{\clw_1} - P_{V_2 \clw_1} = P_{\clw_2} - P_{V_1 \clw_2},
\]
and hence
\[
\tilde{C} = U^* P_{\clw_1} U - U^* P_{V_2 \clw_1} U = U^* P_{\clw_2} U - U^* P_{V_1 \clw_2} U.
\]
By the difference of projection formulae, Theorem \ref{struc}, there exist a projection $R$ on $\mathcal{N}$ and a unitary $w$ on $\mathcal{K}_+$ that commutes with $D$ such that
\begin{equation}\label{eqn1}
U^* P_{\clw_1} U = R \oplus I_{E_1} \oplus 0 \oplus \left[ \begin{array}{cc} \frac{I+D}{2} & \frac{\sqrt{I - D^2}}{2} w\\
w^*\frac{\sqrt{I - D^2}}{2} & \frac{I - D}{2}
\end{array}
\right],
\end{equation}
and
\begin{equation}\label{eqn2}
U^* P_{V_2\clw_1} U = R \oplus 0 \oplus I_{E_{-1}} \oplus \left[ \begin{array}{cc} \frac{I-D}{2} & \frac{\sqrt{I - D^2}}{2} w\\ w^*\frac{\sqrt{I - D^2}}{2} & \frac{I + D}{2}
\end{array}
\right].
\end{equation}
Using the definition of the unitary $U$ in \eqref{eqn Def of U}, we also obtain from the above that:
\begin{equation}\label{eqn3}
P_{\clw_1} = R \oplus I_{E_1} \oplus 0 \oplus \left[ \begin{array}{cc} \frac{I+D}{2} & \frac{\sqrt{I - D^2}}{2} w u^*\\
uw^*\frac{\sqrt{I - D^2}}{2} & u\frac{I - D}{2}u^*
\end{array}
\right],
\end{equation}
and
\begin{equation}\label{eqn4}
P_{V_2\clw_1} = R \oplus 0 \oplus I_{E_{-1}} \oplus \left[ \begin{array}{cc} \frac{I-D}{2} & \frac{\sqrt{I - D^2}}{2} wu^*\\
uw^*\frac{\sqrt{I - D^2}}{2} & u\frac{I + D}{2}u^*
\end{array}
\right].
\end{equation}
Moreover, since $\clw = \clw_1 \oplus V_1 \clw_2$, we have that $I_{\clw} = P_{\clw_1} + P_{V_1 \clw_2}$, and hence $U^* P_{V_1\clw_2} U = I_{\widetilde{\cle}} - U^* P_{\clw_1} U$. Similarly, $U^* P_{\clw_2} U = I_{\widetilde{\cle}} - U^* P_{V_2\clw_1} U$. Therefore, we conclude, by using \eqref{eqn1} and \eqref{eqn2}, that
\begin{equation}\label{missed}
U^* P_{V_1\clw_2} U = R^\perp \oplus 0 \oplus I_{E_{-1}} \oplus \left[ \begin{array}{cc} \frac{I-D}{2} & -\frac{\sqrt{I - D^2}}{2} w \\
-w^*\frac{\sqrt{I - D^2}}{2} & \frac{I + D}{2}
\end{array}
\right],
\end{equation}
and
\begin{equation}\label{missed_1}
U^* P_{\clw_2} U = R^\perp \oplus I_{E_1} \oplus 0 \oplus \left[ \begin{array}{cc} \frac{I+D}{2} & -\frac{\sqrt{I - D^2}}{2} w \\-w^*\frac{\sqrt{I - D^2}}{2} & \frac{I - D}{2} \\
\end{array}
\right].
\end{equation}

\begin{remark}\label{rem rank form}
Particular attention must be paid to isometric pairs with compact defect operators, as they will be utilized frequently in the subsequent sections. Let $(V_1, V_2)$ be an isometric pair. Suppose that $C(V_1, V_2)$ is compact. By \eqref{eqn K+} and \eqref{eqn K-}, it follows that
\[
\clk_+ = \underset{{\lambda \in \sigma (C(V_1, V_2)) \cap (0, 1)}}{\oplus}E_\lambda,
\]
and
\[
\clk_- = \underset{{\lambda \in \sigma (C(V_1, V_2)) \cap (0, 1)}}{\oplus}E_{-\lambda}.
\]
Consequently
\begin{align*}
D = \displaystyle\bigoplus_{\lambda \in \sigma(C(V_1, V_2)) \cap (0, 1)} \lambda I_{E_{\lambda}}.
\end{align*}
Moreover, in this case, it is evident from the description of the unitary $u : \clk_+ \to \clk_-$ (see \eqref{HiNew6}) that \[u(E_\lambda) = E_{-\lambda}\]
for all $\lambda \in \sigma(C(V_1, V_2)) \cap (0, 1)$.
\end{remark}

We are now ready to establish the desired relation between the rank of $C(V_1, V_2)$ and the rank of the cross-commutator $[V_2^*, V_1]$ (see \eqref{eqn def eigenspace 2} for the definition of $E_{\pm 1}$ and Definition \ref{def + part} for the meaning of $\clk_+$).

\begin{theorem}\label{index}
If $(V_1, V_2)$ is an isometric pair, then
\[
\text{rank} C(V_1, V_2) = \text{rank}[V_2^*, V_1] + \dim E_1 + \dim \clk_+,
\]	
where $\clk_+ \subseteq \clh$ is the positive generic part of $(V_1, V_2)$. If, in addition
\[
\dim E_{-1} < \infty,
\]
then
\[
\text{rank} C(V_1, V_2) = 2 \text{rank}[V_2^*, V_1] + \dim E_1 - \dim E_{-1}.
\]
\end{theorem}
\begin{proof}
Observe that
\[
\begin{split}
[V_2^*, V_1]V_1 V_2 & = V_2^* V_1 V_1 V_2 - V_1 V_2^* V_2 V_1
\\
& = 0,
\end{split}
\]
and similarly
\[
V_1^* V_2^* [V_2^*, V_1] = 0.
\]
Since $V = V_1V_2$, we conclude that
\[
[V_2^*, V_1] = 0 \text{ on } \text{ran} V,
\]
and
\[
\text{ran} [V_2^*, V_1] \subseteq \mathcal{W}.
\]
Since $(\text{ran} V)^\perp = \clw$, it follows that
\[
\text{ rank} [V_2^*, V_1] = \text{rank} [V_2^*, V_1]\bigr\rvert_{\clw}.
\]
By Theorem \ref{bcl}, $\clw = \clw_2 \oplus V_2 \clw_1$, and hence
\begin{align*}
\text{ran} \Big([V_2^*, V_1]\bigr\rvert_{\clw}\Big) &= [V_2^*, V_1](\clw)
\\
&= [V_2^*, V_1](\clw_2 \oplus V_2 \clw_1)
\\
&= V_2^* V_1 (\clw_2)
\\
&= \text{ran}\Big(V_2^* P_{V_1 \clw_2}\Big),
\end{align*}
so that
\[
\text{ rank} \Big([V_2^*, V_1]\bigr\rvert_{\clw}\Big) = \text{rank}\Big(V_2^* P_{V_1 \clw_2}\Big).
\]
Since $V_2$ is an isometry, it is clear that
\[
\begin{split}
\text{rank}\Big(V_2^* P_{V_1 \clw_2}\Big) & = \text{rank}\Big(V_2V_2^* P_{V_1 \clw_2}\Big)
\\
& = \text{rank} \Big(P_{\text{ran} V_2} P_{V_1 \clw_2}\Big).
\end{split}
\]
Again, by Theorem \ref{bcl}, we know that $\clw = \clw_2 \oplus V_2 \clw_1$, and hence
\[
\begin{split}
V_2V_2^* \clw & = V_2V_2^*(\clw_2 \oplus V_2 \clw_1)
\\
& = V_2 \clw_1,
\end{split}
\]
so that
\[
\text{rank}(V_2^* P_{V_1 \clw_2})  = \text{ rank}(P_{V_2 \clw_1} P_{V_1 \clw_2}).
\]
On one hand, by \eqref{eqn2} and \eqref{missed}, we have
\[
\begin{split}
\big(U^* P_{V_2\clw_1} U\big) \big(U^* P_{V_1\clw_2} U \big) & = \Big(R \oplus 0 \oplus I_{E_{-1}} \oplus \left[ \begin{array}{cc} \frac{I-D}{2} & \frac{\sqrt{I - D^2}}{2} w\\
w^*\frac{\sqrt{I - D^2}}{2} & \frac{I + D}{2}
\end{array}\right] \Big)
\\
& \qquad \times \Big(R^\perp \oplus 0 \oplus I_{E_{-1}} \oplus \left[ \begin{array}{cc} \frac{I-D}{2} & -\frac{\sqrt{I - D^2}}{2} w \\-w^*\frac{\sqrt{I - D^2}}{2} & \frac{I + D}{2}
\end{array}
\right]\Big)
\\
& = 0 \oplus 0 \oplus I_{E_{-1}} \oplus \left[ \begin{array}{cc} D & 0 \\ 0 & D \\\end{array}
\right] \left[ \begin{array}{cc} -I & 0 \\0 & I \\
\end{array} \right]
\\
& \qquad \times
\left[ \begin{array}{cc} \frac{I-D}{2} & -\frac{\sqrt{I - D^2}}{2} w \\
-w^*\frac{\sqrt{I - D^2}}{2} & \frac{I + D}{2}
\end{array}\right].
\end{split}
\]
Since $D$ is injective, it follows that
\[
\text{rank} \Big(\big(U^* P_{V_2\clw_1} U \big) \big(U^* P_{V_1\clw_2} U\big)\Big) = \dim E_{-1} + \text{rank} \left[ \begin{array}{cc} \frac{I-D}{2} & -\frac{\sqrt{I - D^2}}{2} w \\
-w^*\frac{\sqrt{I - D^2}}{2} & \frac{I + D}{2}
\end{array}\right].
\]
On the other hand, we know that $\tilde{C} = U^* P_{\clw_2} U - U^* P_{V_1 \clw_2} U$. Then \eqref{missed} and \eqref{missed_1} yield
\[
\begin{split}
\tilde{C} & = \Big(R^\perp \oplus I_{E_1} \oplus 0 \oplus \left[ \begin{array}{cc} \frac{I+D}{2} & -\frac{\sqrt{I - D^2}}{2} w\\
-w^*\frac{\sqrt{I - D^2}}{2} & \frac{I - D}{2} \\
\end{array}
\right]\Big)
\\
& \qquad  - \Big(R^\perp \oplus 0 \oplus I_{E_{-1}} \oplus \left[ \begin{array}{cc} \frac{I-D}{2} & -\frac{\sqrt{I - D^2}}{2} w\\ -w^*\frac{\sqrt{I - D^2}}{2} & \frac{I + D}{2} \\
\end{array}
\right]\Big).
\end{split}
\]
This leads us to the setting of Theorem \ref{struc}, and hence by Lemma \ref{rank}, we conclude
\[
\dim \mathcal{K}_+ = \text{rank} \left[ \begin{array}{cc} \frac{I-D}{2} & -\frac{\sqrt{I - D^2}}{2} w \\
-w^*\frac{\sqrt{I - D^2}}{2} & \frac{I + D}{2} \end{array}
\right].
\]
Therefore
\[
\begin{split}
\text{rank} [V_2^*, V_1] & = \text{rank} \big(P_{V_2 \mathcal{W}_1} P_{V_1 \mathcal{W}_2}\big)
\\
& = \text{ rank} \Big(\big( U^* P_{V_2 \mathcal{W}_1} U \big) \big(U^* P_{V_1 \mathcal{W}_2} U\big)\Big)
\\
& = \dim E_{-1} + \dim \mathcal{K}_+,
\end{split}
\]
and hence
\begin{equation}\label{HiNew8}
\text{rank} [V_2^*, V_1] = \dim E_{-1} + \dim \mathcal{K}_+.
\end{equation}
Since $C(V_1, V_2)|_{\clk_+}$ and $C(V_1, V_2)|_{\clk_-}$ are injective, in view of the representation of $C(V_1, V_2)$ as in \eqref{eqn C(V1 V2)}, we have
\[
\text{rank} C(V_1, V_2) = \dim E_1 + \dim E_{-1} + \dim \mathcal{K}_+ + \dim \mathcal{K}_-.
\]
As $\clk_+$ and $\clk_-$ are unitarily equivalent (see the remark preceding \eqref{HiNew6}), we have
\[
\dim \mathcal{K}_+  = \dim \mathcal{K}_-,
\]
and hence
\begin{equation}\label{HiNew9}
\text{rank} C(V_1, V_2) = \dim E_1 + \dim E_{-1} + 2 \dim \mathcal{K}_+.
\end{equation}
It now follows immediately from \eqref{HiNew8} that
\[
\text{rank}C(V_1, V_2) = \text{rank}[V_2^*, V_1] + \dim E_1 + \dim \clk_+.
\]
This completes the proof of the first part of the theorem. Suppose further that
\[
\dim E_{- 1} < \infty.
\]
From \eqref{HiNew8}, it follows that
\[
2 \text{rank}[V_2^*, V_1] = 2 \dim E_{-1} + 2 \dim \clk_+,
\]
and hence
\[
2 \text{rank}[V_2^*, V_1] - \dim E_{-1} = \dim E_{-1} + 2 \dim \clk_+.
\]
Finally, it follows from \eqref{HiNew9} that, by substituting the value of $\dim E_{-1} + 2 \dim \clk_+$ obtained above
\[
\text{rank}C(V_1, V_2) = 2 \text{rank}[V_2^*, V_1] + \dim E_1 - \dim E_{-1}.
\]
This completes the proof of the theorem.
\end{proof}

From the spectral theorem for compact self-adjoint operators and the second part of the preceding theorem, it follows immediately that:

\begin{coro}\label{cor rank formula}
Let $(V_1, V_2)$ be an isometric pair such that $C(V_1, V_2)$ is compact. Then
\[
\text{rank} C(V_1, V_2) = 2 \text{rank} [V_2^*, V_1] + \dim E_1 - \dim E_{-1}.
\]
\end{coro}

We will apply this particular version of the rank formula in the upcoming analysis. As an immediate consequence of the above theorem, we also obtain that:

\begin{coro}\label{HiNew10}
Let $(V_1, V_2)$ be an isometric pair such that
\[
\dim E_1(C(V_1, V_2)) < \infty.
\]
Assume that $[V_2^*, V_1]$ is a finite-rank operator. Then $C(V_1, V_2)$ is a finite-rank operator.
\end{coro}
\begin{proof}
Since $\text{rank}[V_2^*, V_1] < \infty$, it follows from \eqref{HiNew8} that
\[
\dim E_{-1}, \dim \clk_+ < \infty.
\]
Moreover, since, by assumption, $\dim E_1 < \infty$, it follows from Theorem \ref{index} that
\[
\text{rank}C(V_1, V_2) = \text{rank}[V_2^*, V_1] + \dim E_1 + \dim \clk_+ < \infty,
\]
which completes the proof of the corollary.
\end{proof}

We conclude the section with a dimension formula that is of independent interest.

\begin{proposition}
Let $(V_1, V_2)$ be an isometric pair. Suppose
\[
\big(\ker C(V_1, V_2)\big)^\perp = \mathcal{W}.
\]
Then
\begin{align*}
\text{dim} \mathcal{W}_1 = \text{dim}\mathcal{W}_2
\end{align*}
where $\mathcal{W}_i = \ker V_i^*$ for $i = 1, 2$.
In particular, if $\mathcal{W}$ is finite-dimensional, then $\text{dim} \mathcal{W}$ is even.
\end{proposition}

\begin{proof}
As $\big(\ker C(V_1, V_2)\big)^\perp = \clw$, we have that $\mathcal{N} = \mathcal{W} \ominus \big(\ker C(V_1, V_2)\big)^\perp = 0$, and consequently, by \eqref{eqn1} and \eqref{missed_1}, it follows that
\begin{equation*}
U^* P_{\clw_2} U = I_{E_1} \oplus 0 \oplus \left[ \begin{array}{cc} \frac{I+D}{2} & -\frac{\sqrt{I - D^2}}{2} w \\
-w^*\frac{\sqrt{I - D^2}}{2} & \frac{I - D}{2} \\
\end{array}
\right],
\end{equation*}
and
\begin{equation*}
U^* P_{\clw_1} U = I_{E_1} \oplus 0 \oplus \left[ \begin{array}{cc} \frac{I+D}{2} & \frac{\sqrt{I - D^2}}{2} w\\
w^*\frac{\sqrt{I - D^2}}{2} & \frac{I - D}{2} \\
\end{array}
\right],
\end{equation*}
where $w$ is a unitary on $\mathcal{K}_+$ that commutes with $D$. By Lemma \ref{rank}, we have
\[
\begin{split}
\dim \mathcal{K}_+ & =
\text{rank}\left[ \begin{array}{cc} \frac{I+D}{2} & \frac{\sqrt{I - D^2}}{2} w\\
w^*\frac{\sqrt{I - D^2}}{2} & \frac{I - D}{2} \\
\end{array}
\right]
\\
& = \text{rank}\left[ \begin{array}{cc} \frac{I+D}{2} & -\frac{\sqrt{I - D^2}}{2} w\\
-w^*\frac{\sqrt{I - D^2}}{2} & \frac{I - D}{2} \\
\end{array}
\right],
\end{split}
\]
and hence
\[
\begin{split}
\text{dim} \clw_1 & = \text{rank}(U^* P_{\clw_1} U)
\\
& = \dim E_1 + \dim \clk_+
\\
& = \text{rank}( U^* P_{\clw_2} U)
\\
& = \dim \clw_2.
\end{split}
\]
In particular, if $\text{dim}\clw < \infty$, then
\[
\begin{split}
\dim \clw & = \dim \clw_1 + \dim \clw_2
\\
& = 2 \times \dim \clw_1
\\
&= 2 \times \dim \clw_2,
\end{split}
\]	
which completes the proof of the proposition.
\end{proof}

The challenge of formulating rank identities for invariant subspaces is well recognized as a complex and difficult field of study. It not only reveals the structure of invariant subspaces but also entails identities involving numbers. We anticipate that the rank formula established in Theorem \ref{index} possesses inherent value and will be applicable in various different contexts. For instance, consider a closed invariant subspace $\clm$ of $H^2(\D^2)$ (see \eqref{eqn: inv sub Hardy}). As in \eqref{eqn rest of Hardy inv}, define the restriction operators
\[
R_z = M_z|_{\clm} \text{ and } R_w = M_w|_{\clm}.
\]
Clearly, $(R_z, R_w)$ is an isometric pair on $\clm$. Moreover
\[
R_z R_w = M_z M_w|_{\clm},
\]
and hence $(R_z, R_w)$ is a BCL pair on $\clm$. Consequently, Theorem \ref{index} applies to $\clm$ and hence invariant subspaces of $H^2(\D^2)$. In the present context, the rank formula in Theorem \ref{index} should be compared with the rank formula of Yang \cite[Theorem 2.7]{RY JFA}.

In the literature, there appear to be very intricate rank formulae for Hilbert-Schmidt invariant (as well as co-invariant) subspaces of $H^2(\D^2)$ (see \cite{CDS} and references therein). 

\section{On $3$-finite pairs}\label{sec properties of 3 finite}

The purpose of this section is to isolate key properties of irreducible $3$-finite pairs. Some of the results do not require all the assumptions of $3$-finite pairs, and we will point out the needed properties for such results. For the convenience of the subsequent discussion, we shall include an additional stratum of notation: Given a separable Hilbert space $\mathcal{H}$, denote by $B_{\mathcal{H}}$ the set of all ordered orthonormal bases of $\clh$. That is
\[
B_{\clh} = \{\{e_j: j \in \Lambda\}:  \{e_j: j \in \Lambda\} \text{ is an orthonormal basis of } \mathcal{H}\},
\]
where $\Lambda$ denotes a countable set. Let $(V_1, V_2)$ be an isometric pair. For the reader's convenience, we recall that $V = V_1 V_2$, and
\[
\mathcal{W} = \ker V^* \text{ and } \mathcal{W}_i = \ker V_i^*,
\]
for $i = 1, 2$. Recall also that
\[
E_\lambda := \{f \in \clh: C(V_1, V_2)f = \lambda f\},
\]
for all $\lambda \in \mathbb{C}$ (see \eqref{eqn def eigenspace 2}). We begin with a useful property of BCL pairs, stated in the following well-known lemma (see, for instance, \cite[Proposition 4.1]{HQY2015} for a proof). 
\begin{lemma}\label{intersection}
Let $(V_1, V_2)$ be a BCL pair. Then
\[
\clw_1 \cap \clw_2 = E_1.
\]
\end{lemma}

The following lemma, in particular, shows that the range of $[V_2^*, V_1]$ is contained in $\mathcal{W}_1 \cap \mathcal{W}_2$ whenever $[V_2^*, V_1]$ is normal.

\begin{lemma}\label{contain}
Let $(V_1, V_2)$ be a BCL pair. If $[V_2^*, V_1]$ is normal, then
\[
\text{ran }[V_2^*, V_1] = \text{ran }[V_1^*, V_2] \subseteq E_1,
\]
and
\[
[V_2^*, V_1]|_{ E_1^\perp} = [V_1^*, V_2]|_{ E_1^\perp} = 0.
\]
\end{lemma}
\begin{proof}
The normality of $[V_2^*, V_1]$ yields
\[
\text{ran} [V_2^*, V_1] = \text{ran} [V_2^*, V_1]^* = \text{ran} [V_1^*, V_2].
\]
Observe that
\[
V_2^*(V_1^* V_2 - V_2 V_1^*) = 0 = V_1^*(V_2^*V_1 - V_1 V_2^*),
\]
that is
\[
 V_2^* [V_1^*, V_2] = 0 = V_1^* [V_2^*, V_1].
\]
Clearly $\text{ran} [V_2^*, V_1] \subseteq \mathcal{W}_1$ and $\text{ran} [V_1^*, V_2] \subseteq \mathcal{W}_2$, and hence, by Lemma \ref{intersection}, we conclude that
\[
\text{ran} [V_2^*, V_1] = \text{ran} [V_1^*, V_2] \subseteq \mathcal{W}_1 \cap \mathcal{W}_2 = E_1.
\]
For the second assertion, suppose $g \in E_1^\perp$ and set
\[
h = [V_2^*, V_1]g.
\]
Note that
\[
||h||^2 = \langle[V_1^*, V_2][V_2^*, V_1]g, g \rangle.
\]
But by the first assertion of this lemma, it is clear that
\[
[V_1^*, V_2][V_2^*, V_1]g \in E_1,
\]
and consequently, $||h||^2 = 0$, that is, $h = 0$. This completes the proof of the lemma.
\end{proof}

From now on, throughout the section, we will deal with irreducible $3$-finite pairs (see Definition \ref{def irred} for irreducible pairs).

\begin{proposition}\label{dime1}
Let $(V_1, V_2)$ be an irreducible $3$-finite pair. Then
\[
\dim E_1 = 1 \text{ and } \dim E_{-1} = 0.
\]
\end{proposition}
\begin{proof}
By virtue of Lemmas \ref{intersection} and \ref{contain}, we already know that
\[
\text{ran} [V_2^*, V_1] = \text{ran} [V_1^*, V_2] \subseteq \mathcal{W}_1 \cap \mathcal{W}_2 = E_1,
\]
and
\[
[V_2^*, V_1]|_{E_1^\perp} = [V_2^*, V_1]^*|_{E_1^\perp} = 0.
\]
In particular, $[V_2^*, V_1]|_{E_1}$ is a normal operator on $E_1$. If possible, let
\[
\dim E_1 > 1,
\]
and let $\{f, g\}$ be an orthonormal set in $E_1$ consisting of eigen vectors of the normal operator $[V_2^*, V_1]|_{E_1}$. Set
\[
\mathcal{S} = \overline{\text{span}} \{V_1^m V_2^n f : m, n \geq 0\}.
\]
We claim that $\mathcal{S}$ reduces $(V_1, V_2)$. This would contradict the fact that $(V_1, V_2)$ is irreducible. Clearly, $\mathcal{S}$ is invariant under $V_1$ and $V_2$. Therefore, to prove the claim, it suffices to show that
\[
V_1^*V_2^n f, V_2^*V_1^n f \in \mathcal{S},
\]
for all $n \geq 1$. We only prove that $V_2^*V_1^n f \in \mathcal{S}$ for all $n \geq 1$ (the proof of the remaining case is similar).  We prove this by induction. For $n=1$, since $f \in \mathcal{W}_1 \cap \mathcal{W}_2$ and $f$ is an eigen vector of $[V_2^*, V_1]$, it follows that
\[
V_2^* V_1f = [V_2^*, V_1]f = \alpha f,
\]
for some scalar $\alpha$, and hence $V_2^*V_1f \in  \mathcal{S}$. Thus, the result is true for $n = 1$. Now, suppose that the result is true for $m \geq 1$, that is, $V_2^* V_1^m f \in \mathcal{S}$. Write
\begin{align*}
V_2^* V_1^{m+1}f = [V_2^*, V_1]V_1^{m}f + V_1 V_2^* V_1^{m}f.
\end{align*}
Since $m \geq 1$, it is clear, in particular, that $V_1^{m}f \in (\mathcal{W}_1 \cap \mathcal{W}_2)^\perp$. Note that Lemma \ref{contain} also implies that
\[
[V_2^*, V_1]|_{(\mathcal{W}_1 \cap \mathcal{W}_2)^\perp} = 0.
\]
Therefore, $[V_2^*, V_1]V_1^{m}f = 0$, and hence
\begin{equation}\label{impeqn}
V_2^* V_1^{m+1}f = V_1 V_2^* V_1^{m}f.
\end{equation}
Since $\mathcal{S}$ is invariant under $V_1$, and by the induction hypothesis $V_2^*V_1^mf \in \mathcal{S}$, it follows that
\[
V_1 V_2^* V_1^{m}f \in \mathcal{S},
\]
that is,
\[
V_2^* V_1^{m+1} f \in \mathcal{S}.
\]
Thus, the result is true for $n = m+1$. Hence, by the principle of mathematical induction, the result is true for all $n \geq 1$. This proves the claim and then the fact that $\dim E_1 \leq 1$. Since $\text{rank}C(V_1, V_2) = 3$, the desired equality $\dim E_1 =1$ follows immediately by an appeal to Lemma \ref{contain} and the rank formula in Corollary \ref{cor rank formula}.

\noindent Now we prove that $\dim E_{-1} = 0$. Since $\dim E_1 = 1$ and $\text{ran} [V_2^*, V_1] \subseteq E_1$, it follows that
\[
\text{rank}[V_2^*, V_1] \leq 1.
\]
As $\text{rank} C(V_1, V_2) = 3$, it follows by the rank formula in Corollary \ref{cor rank formula} that
\[
\begin{split}
3 & = \text{rank} C(V_1, V_2)
\\
& = 2 \text{rank} [V_2^*, V_1] + \dim E_1 - E_{-1}
\\
& \leq 2 + 1 - \dim E_{-1}
\\
& = 3 - \dim E_{-1}.
\end{split}
\]
This proves that $\dim E_{-1} = 0$.
\end{proof} 

The following observation is now straight:

\begin{coro}\label{rankcross}
Let $(V_1, V_2)$ be an irreducible $3$-finite pair. Then
\[
\text{rank} [V_2^*, V_1] = 1.
\]
In particular, $\text{ran} [V_2^*, V_1] = E_1$.
\end{coro}

We continue with an irreducible $3$-finite pair $(V_1, V_2)$. Recall that the symbol $\cong$ stands for unitary equivalance of operators.

\begin{proposition}\label{rank3defect}
Let $(V_1, V_2)$ be an irreducible $3$-finite pair. Then there is a unique $\lambda_{(V_1, V_2)} \in (0, 1)$ such that
\[
\sigma(C(V_1, V_2)) \cap (0, 1) = \{\lambda_{(V_1, V_2)}\}.
\]
Moreover
\[
C(V_1, V_2)|_{(\ker C(V_1, V_2))^\perp} \cong D_{\lambda_{(V_1, V_2)}},
\]
where $D_{\lambda_{(V_1, V_2)}}$ is the diagonal matrix
\[
D_{\lambda_{(V_1, V_2)}} = \begin{bmatrix}
1 &  & \\
& \lambda_{(V_1, V_2)} &  \\
& & -\lambda_{(V_1, V_2)} \\
\end{bmatrix}.
\]
\end{proposition}
\begin{proof}
By Proposition \ref{dime1}, we know that $\dim E_1 = 1, \dim E_{-1} = 0$. Since $\text{rank} C(V_1, V_2) = 3$, it is evident from Theorem \ref{structure} that there exists $\lambda_{(V_1, V_2)} \in (0, 1)$ such that
\[
\sigma(C(V_1, V_2)) \cap (0, 1) = \{\lambda_{(V_1, V_2)}\},
\]
and $C(V_1, V_2)|_{(\ker C(V_1, V_2))^\perp}$ is unitarily equivalent to $D_{\lambda_{(V_1, V_2)}}$.
\end{proof}

Let $(V_1, V_2)$ be an irreducible $3$-finite pair. Then
\[
\ker(C(V_1, V_2))^\perp = E_1 \oplus E_{\lambda_{(V_1, V_2)}} \oplus E_{-\lambda_{(V_1, V_2)}},
\]
where $\sigma(C(V_1, V_2)) \cap (0, 1) = \{\lambda_{(V_1, V_2)}\}$. As usual, we set
\[
\mathcal{N} = \mathcal{W} \ominus \ker(C(V_1, V_2))^\perp.
\]
Then
\[
\mathcal{W} = E_1 \oplus E_{\lambda_{(V_1, V_2)}} \oplus E_{- \lambda_{(V_1, V_2)}} \oplus \mathcal{N},
\]
and with respect to this decomposition, we have
\[
C(V_1, V_2)|_{\clw} = P_{\mathcal{W}_1} - P_{V_2 \mathcal{W}_1} = \begin{bmatrix}
1 &  &  &\\
& \lambda_{(V_1, V_2)} &  &\\
& & -\lambda_{(V_1, V_2)} & \\
& & & 0_{\mathcal{N}}
\end{bmatrix}.
\]
On the other hand, since $[V_2^*, V_1]$ is normal, and by Corollary \ref{rankcross}, $\text{rank}[V_2^*, V_1] = 1$ with $\text{ran}[V_2^*, V_1] = E_1$, there exist a nonzero scalar $\beta_{(V_1, V_2)}$ such that
\[
\sigma([V_2^*, V_1]) \setminus \{0\} = \{\beta_{(V_1, V_2)}\},
\]
and
\[
[V_2^*, V_1] f = \beta_{(V_1, V_2)} f \qquad (f \in E_1).
\]
Fix a unit vector $f_{(V_1, V_2)} \in E_1$. Clearly
\[
\{f_{(V_1, V_2)}\} \in B_{E_1}.
\]
Thus, in summary, we have the following:
\[
\begin{cases}
\sigma(C(V_1, V_2)) \cap (0, 1) = \{\lambda_{(V_1, V_2)}\}
\\
\sigma([V_2^*, V_1]) \setminus \{0\} = \{\beta_{(V_1, V_2)}\}
\\
\{f_{(V_1, V_2)}\} \in B_{E_1}.
\end{cases}
\]
At this juncture, we standardize some notation that will be used throughout the rest of the paper. Given an irreducible $3$-finite pair $(V_1, V_2)$, we set
\begin{equation}\label{eqn 3 notations 1}
\begin{cases}
\beta_{(V_1, V_2)} = \text{unique nonzero eigenvalue of } [V_2^*, V_1]
\\
\lambda_{(V_1, V_2)} = \text{unique eigenvalue of } C(V_1, V_2) \text{ in } (0, 1)
\\
f_{(V_1, V_2)} = \text{a unit vector in } E_1.	
\end{cases}
\end{equation}
Furthermore, for notational convenience, if the pair $(V_1, V_2)$ is clear from the context, then we simply write
\begin{equation}\label{eqn 3 notations 2}
\lambda = \lambda_{(V_1, V_2)}, \; \beta = \beta_{(V_1, V_2)}, \text{ and } f = f_{(V_1, V_2)}.
\end{equation}
Also, for any $g, h$ in $\mathcal{H}$, we denote by $g \otimes h$ the rank one operator on $\mathcal{H}$ where
\[
(g \otimes h)(k) = \langle k, h \rangle g \qquad (k \in \mathcal{H}).
\]
Thus, in the case of the present scenario, we can write
\[
[V_2^*, V_1] = \beta f \otimes f,
\]
and hence
\begin{equation}\label{eqn cross and beta}
[V_1^*, V_2] = \bar{\beta} f \otimes f.
\end{equation}
Thus, we have
\[
[V_2^*, V_1]f = V_2^* V_1 f = \beta f,
\]
and
\[
[V_1^*, V_2]f = V_1^* V_2 f = \bar{\beta}f.
\]
Therefore
\begin{equation}\label{crosscomm}
\begin{cases}
V_2^* V_1 f = \beta f
\\
V_1^* V_2 f = \bar{\beta}f.
\end{cases}
\end{equation}
Since $f \in \clw_1 \cap \clw_2$, it follows that $C(V_1, V_2)f = f$, and hence,
\[
\{f, e_\lambda, e_{-\lambda}\} \in B_{E_1 \oplus E_\lambda \oplus E_{-\lambda}},
\]
where
\[
\{e_\lambda\} \in B_{E_\lambda} \text{ and } \{e_{-\lambda}\} \in B_{E_{-\lambda}}.
\]
Set
\[
E^\lambda = E_\lambda \oplus E_{-\lambda}.
\]
Now we follow the construction as given in Section \ref{sec rank formula}. By applying Remark \ref{rem rank form}, in this particular situation, we find
\[
\clk_{\pm} = E_{\pm \lambda},
\]
and
\[
D = \lambda I_{E_\lambda}.
\]
Therefore, following the representation of $\clw$ as in \eqref{eqn rep of W}, we have
\[
\mathcal{W} = E_1 \oplus E^\lambda \oplus \mathcal{N},
\]
and so, by \eqref{eqn3} and \eqref{eqn4}, it follows that
\[
P_{\clw_1} = I_{E_1} \oplus \left[ \begin{array}{cc} \frac{I+D}{2} & \frac{\sqrt{I - D^2}}{2} w u^*\\
uw^*\frac{\sqrt{I - D^2}}{2} & u\frac{I - D}{2}u^*
\end{array}
\right] \oplus Q,
\]
and
\[
P_{V_2\clw_1} = 0 \oplus \left[ \begin{array}{cc} \frac{I-D}{2} & \frac{\sqrt{I - D^2}}{2} wu^*\\
uw^*\frac{\sqrt{I - D^2}}{2} & u\frac{I + D}{2}u^*
\end{array}
\right] \oplus Q,
\]
where $u: E_{\lambda} \rightarrow E_{-\lambda}$ and $w: E_{\lambda} \rightarrow E_{\lambda}$ are unitary operators, and $Q: \mathcal{N} \rightarrow \mathcal{N}$ is a projection. Consequently
\[
w u^*: E_{-\lambda} \rightarrow E_{\lambda},
\]
is a unitary, and so, there exists a unimodular constant $\alpha$ such that
\[
wu^* (e_{-\lambda}) = \alpha e_{\lambda}.
\]
Define
\[
Q_{\lambda} = \left[ \begin{array}{cc} \frac{I+D}{2} & \frac{\sqrt{I - D^2}}{2} w u^*\\
uw^*\frac{\sqrt{I - D^2}}{2} & u\frac{I - D}{2}u^*
\end{array}
\right],
\]
and
\[
\tilde{Q}_\lambda =  \left[ \begin{array}{cc} \frac{I-D}{2} & \frac{\sqrt{I - D^2}}{2} wu^*\\
uw^*\frac{\sqrt{I - D^2}}{2} & u\frac{I + D}{2}u^*
\end{array}
\right].
\]
Therefore, $Q_\lambda$ and $\tilde{Q}_\lambda$ are projections defined on $E^\lambda$. The  matrix representations of $Q_\lambda$ and $\tilde{Q}_\lambda$ with respect to
\[
\{e_\lambda, e_{-\lambda}\} \in B_{E^\lambda},
\]
are given by
\[
Q_\lambda = \frac{1}{2} \left[ \begin{array}{cc}
1 + \lambda & \alpha(1-\lambda^2)^{\frac{1}{2}} \\
\bar{\alpha} (1-\lambda^2)^{\frac{1}{2}} & 1 - \lambda\\
\end{array}
\right],
\]
and
\[
\tilde{Q}_\lambda = \frac{1}{2} \left[ \begin{array}{cc}
1 - \lambda & \alpha(1-\lambda^2)^{\frac{1}{2}} \\
\bar{\alpha} (1-\lambda^2)^{\frac{1}{2}} & 1 + \lambda\\
\end{array}
\right].
\]
Thus, with respect to the decomposition $\mathcal{W} = E_1 \oplus E^\lambda \oplus \mathcal{N}$, we have
\begin{equation}\label{eqn5}
P_{\mathcal{W}_1} = \begin{bmatrix}
I_{E_1} &  &  \\
& Q_\lambda &  \\
& & Q
\end{bmatrix} \text{ and } P_{V_2\mathcal{W}_1} = \begin{bmatrix}
0 &  &  \\
& \tilde{Q}_{\lambda} &  \\
& & Q
\end{bmatrix},
\end{equation}
where $Q_\lambda$ and $\tilde{Q}_\lambda$ are given as above. As
\[
\mathcal{W} = \mathcal{W}_1 \oplus V_1 \mathcal{W}_2 = \mathcal{W}_2 \oplus V_2 \mathcal{W}_1,
\]
it follows that $P_{\mathcal{W}_2} = I_\mathcal{W} - P_{V_2 \mathcal{W}_1}$ and $P_{V_1\mathcal{W}_2} = I_\mathcal{W} - P_{\mathcal{W}_1}$, and hence,
\begin{equation}\label{eqn6}
P_{\mathcal{W}_2} =
\begin{bmatrix}
I_{E_1} &  &  \\
& \tilde{Q}_\lambda^\perp &  \\
& & Q^\perp
\end{bmatrix} \text{ and } P_{V_1\mathcal{W}_2} = \begin{bmatrix}
0 &  &  \\
& Q_\lambda^\perp &  \\
& & Q^\perp
\end{bmatrix}.
\end{equation}
Finally, to find a suitable orthonormal basis of $E^\lambda$, we define
\begin{equation}\label{eqn f1 part I}
f_1 = \sqrt{\frac{1+\lambda}{2}} e_\lambda + \bar{\alpha}\sqrt{\frac{1-\lambda}{2}} e_{-\lambda}\text{ and } f_4 = \sqrt{\frac{1+\lambda}{2}} e_\lambda - \bar{\alpha}\sqrt{\frac{1-\lambda}{2}} e_{-\lambda},
\end{equation}
and
\[
f_2 = \frac{f_1 - \lambda f_4}{\sqrt{1 - \lambda^2}} \text{ and } f_3 = \frac{f_4 - \lambda f_1}{\sqrt{1 - \lambda^2}}.
\]
Then, as easy computation reveals that
\begin{equation}\label{eqn f1 part II}
\begin{cases}
& f_2 = \sqrt{\frac{1-\lambda}{2}} e_\lambda + \bar{\alpha}\sqrt{\frac{1+\lambda}{2}} e_{-\lambda}
  \\
  & f_3 = \sqrt{\frac{1-\lambda}{2}} e_\lambda - \bar{\alpha}\sqrt{\frac{1+\lambda}{2}} e_{-\lambda}.
\end{cases}
\end{equation}

At this point, we recall the following useful result concerning the orthonormal basis of the range of projections of our interest (see \cite[Lemma 3.2]{DSPS}):

\begin{lemma}\label{range}
Let $\mathcal{H}$ and $\mathcal{K}$ be Hilbert spaces, $U : \mathcal{H} \rightarrow \mathcal{K}$ a unitary operator, and let $\lambda \in [-1, 1]$. Define the projection $P : \mathcal{H} \oplus \mathcal{K} \rightarrow \mathcal{H} \oplus \mathcal{K}$ by
\begin{equation*}
P = \left[ \begin{array}{cc}
\frac{1+\lambda}{2}I_{\mathcal{H}} & \frac{\sqrt{1-\lambda^2}}{2} U^*
\\
\frac{\sqrt{1-\lambda^2}}{2} U & \frac{1-\lambda}{2}I_{\mathcal{K}}
\\
\end{array}
\right].
\end{equation*}
If $\{ e_i : i \in \Lambda \} \in B_{\mathcal{H}}$, then $\Big\{ \sqrt{\frac{1+\lambda}{2}} e_i \oplus \sqrt{\frac{1-\lambda}{2}} Ue_i : i \in \Lambda \Big \} \in B_{\ran P}$.
\end{lemma}

Returning to our setting, in particular, we have the following:
	
\begin{coro}\label{fi}
Let $\{f_i\}_{i=1}^4$ be as in \eqref{eqn f1 part I} and \eqref{eqn f1 part II}. Then
\[
\text{ran} Q_\lambda = \mathbb{C} f_1, \text{ran} \tilde{Q}_\lambda = \mathbb{C} f_2, \text{ran} {Q}_\lambda^\perp = \mathbb{C} f_3, \text{ and } \text{ran} \tilde{Q}_\lambda^\perp = \mathbb{C}f_4.
\]
In particular
\[
\{f_1, f_3\}, \{f_2, f_4\} \in B_{E^\lambda}.
\]
\end{coro}

By Theorem \ref{bcl}, we know that $(\mathcal{W}, U, P_{\mathcal{W}_1})$ is a BCL triple associated to $(V_1, V_2)$, where the unitary $U$ on $\clw$ is given by
\begin{equation}\label{unitarydef}
U = \begin{bmatrix}
V_2|_{\mathcal{W}_1} & \\
& V_1^*|_{V_1 \mathcal{W}_2}
\end{bmatrix} : \mathcal{W}_1 \oplus V_1 \mathcal{W}_2 \to V_2\mathcal{W}_1 \oplus \mathcal{W}_2.
\end{equation}
Our goal is now to reveal the action of $U$ on a basis of $\mathcal{W}$. We need some preparatory calculations. It follows immediately from the representations of $P_{\clw_i}$ and $P_{V_j \clw_i}$, $i \neq j$, $i,j=1,2$ (see \eqref{eqn5} and \eqref{eqn6}) that
\[
\mathcal{W}_1 = E_1 \oplus \text{ran} Q_\lambda \oplus \text{ran} Q, \; V_2 \mathcal{W}_1 = \text{ran} \tilde{Q}_\lambda \oplus \text{ran} Q,
\]
and
\[
\mathcal{W}_2 = E_1 \oplus \text{ran} \tilde{Q}_\lambda^\perp \oplus \text{ran} Q^\perp, \; V_1 \mathcal{W}_2 = \text{ran} Q_\lambda^\perp \oplus \text{ran} Q^\perp.
\]
Consequently, by Corollary \ref{fi}, we have
\begin{equation}\label{eqn7}
\begin{aligned} &\mathcal{W}_1  = E_1 \oplus \mathbb{C} f_1 \oplus \text{ran} Q,\\
&V_2 \mathcal{W}_1 = \mathbb{C} f_2 \oplus \text{ran}Q, \\
& \mathcal{W}_2 = E_1 \oplus \mathbb{C} f_4 \oplus \text{ran} Q^\perp, \\ &V_1 \mathcal{W}_2 = \mathbb{C} f_3 \oplus \text{ran} Q^\perp.
\end{aligned}
\end{equation}
Now we are ready to explore the action of the unitary $U$ on $\mathcal{W}$.

\begin{lemma}\label{Uaction}
Let $(V_1, V_2)$ be an irreducible $3$-finite pair. With notations as above, we have the following:
\begin{itemize}
\item[(a)] $U f_3 \in \mathbb{C} f$,
\item[(b)] $U f \in \mathbb{C} f_2$,
\item[(c)] $Uf_1 \in \text{ran}Q$,
\item[(d)] $U^*f_4 \in \text{ran}Q^\perp$,
\item[(e)] $U^* \big(\text{ran}Q^\perp\big) \subset \text{ran}Q^\perp$, and $U\big(\text{ran}Q\big) \subset \text{ran}Q$.
\end{itemize}
\end{lemma}
\begin{proof}
First, we prove part (a). Since $f \in \mathcal{W}_1$, it follows that
\[
Uf = V_2 f \in V_2 \mathcal{W}_1.
\]
By \eqref{eqn7}, we know that $V_2 \mathcal{W}_1 = \mathbb{C} f_2 \oplus \text{ran}Q$, and hence, there exists $g \in \text{ran}Q$ such that
\[
V_2 f = \langle V_2 f, f_2 \rangle f_2 + g.
\]
As $\text{ran}Q \subset \mathcal{W}_1$, applying $V_1^*$ to both sides of the preceding equation and then using \eqref{crosscomm}, we see that
\[
\bar{\beta} f = \langle V_2 f, f_2 \rangle V_1^* f_2.
\]
Now we compute $V_1^* f_2$. By Corollary \ref{fi}, we know that $\{f_1, f_3 \} \in B_{E^\lambda}$. As $f_2 \in E^\lambda (= E_\lambda \oplus E_{-\lambda})$, we have
\[
f_2 = \langle f_2, f_1 \rangle f_1 + \langle f_2, f_3 \rangle f_3.
\]
By \eqref{eqn7}, we know that $f_1 \in \mathcal{W}_1$. Applying $V_1^*$ to both sides of the preceding identity we obtain that
\[
V_1^* f_2 = \langle f_2, f_3 \rangle V_1^* f_3.
\]
Since $f_3 \in V_1\mathcal{W}_2$ (see \eqref{eqn7}), the definition of $U$ as given by \eqref{unitarydef} yields
\[
V_1^* f_2 = \langle f_2, f_3 \rangle U f_3,
\]
and hence
\[
\bar{\beta} f = \langle V_2 f, f_2 \rangle \langle f_2, f_3 \rangle U f_3.
\]
It is clear from \eqref{eqn f1 part II}  and the fact that $\alpha$ is a unimodular constant that $\langle f_2, f_3 \rangle = -\lambda$. Finally, by the definition of $U$, it follows that $V_2 f = Uf$ and hence
\[
U f_3 = -\frac{\bar{\beta}}{\lambda \langle U f, f_2 \rangle}f.
\]
For (b), we observe, similarly, that $U^* f = V_1 f\in V_1 \mathcal{W}_2$. By \eqref{eqn7}, we know that
\[
V_1 \mathcal{W}_2 = \mathbb{C}f_3 \oplus \text{ran} Q^\perp,
\]
and hence, there exists $h \in \text{ran}Q^\perp$ such that
\[
V_1 f = \langle V_1 f, f_3 \rangle f_3 + h.
\]
Now we follow the steps in (a) precisely to conclude that
\begin{equation}\label{formulauf2}
U^* f_2 =  -\frac{\beta}{\lambda \langle U^* f, f_3 \rangle} f.
\end{equation}
Next, we proceed to prove (c). As $E_1 (= \mathbb{C} f)$ and $\text{ran} Q_\lambda( = \mathbb{C} f_1)$ are orthogonal, we have that	$\langle f, f_1 \rangle = 0$. Therefore
\[
\langle Uf, Uf_1 \rangle = 0.
\]
Since $Uf \in \mathbb{C} f_2$ (see part (b)), it follows that $\langle Uf_1, f_2 \rangle = 0$. As $f_1 \in \mathcal{W}_1$, \eqref{unitarydef} together with \eqref{eqn7} imply
\[
Uf_1 \in V_2 \mathcal{W}_1 = \mathbb{C} f_2 \oplus \text{ran}Q.
\]
Therefore, $\langle Uf_1, f_2\rangle = 0$ implies that $Uf_1 \in \text{ran}Q$. Since the proof of (d) is analogous to that of (c), the specifics are omitted. Finally, we turn to (e). Clearly \eqref{eqn7} implies $E_1 (= \mathbb{C} f)$ and $\text{ran}Q^\perp$ are orthogonal subspaces of $\mathcal{W}_2$. Consequently
\[
U^* f \perp U^*(\text{ran}Q^\perp).
\]
By part (a), $U^*f \in \mathbb{C}f_3$ and hence
\[
f_3 \perp U^*(\text{ran}Q^\perp).
\]
The definition of $U$ as in \eqref{unitarydef} together with \eqref{eqn7} yield that
\[
U^*(\text{ran}Q^\perp) \subseteq U^* \mathcal{W}_2 = V_1 \mathcal{W}_2 = \mathbb{C} f_3 \oplus \text{ran}Q^\perp,
\]
and consequently, $U^*(\text{ran }Q^\perp) \subseteq \text{ran}Q^\perp$. The proof of $U(\text{ran }Q) \subseteq \text{ran}Q$ is similar.
\end{proof}

The following theorem provides a summary of the major findings concerning irreducible $3$-finite pairs of isometries so far. The final claim is an addition that states that the absolute value of the nonzero eigenvalue of the cross-commutator is same as the second largest eigenvalue of the defect operator.

\begin{theorem}\label{alphalambda}
Let $(V_1, V_2)$ be an irreducible $3$-finite pair of isometries. Then
\begin{enumerate}
\item $\text{rank}[V_2^*, V_1] = 1$ and $\text{ran}[V_2^*, V_1] = E_1$.
\item $E_{-1} = \{0\}$.
\item There exist unique $\beta_{(V_1, V_2)} \in \mathbb{C} \setminus \{0\}$ and a unit vector $f_{(V_1, V_2)} \in E_1$ such that
\[
E_1 = \mathbb{C} f_{(V_1, V_2)},
\]
and
\[
[V_2^*, V_1] f_{(V_1, V_2)} = \beta_{(V_1, V_2)} f_{(V_1, V_2)}.
\]
\item There is a unique $\lambda_{(V_1, V_2)} \in (0,1)$ such that
\[
\sigma(C(V_1, V_2)) \cap (0, 1) = \{\lambda_{(V_1, V_2)}\},
\]
and
\[
C(V_1, V_2)|_{(\ker C(V_1, V_2))^\perp} \cong \begin{bmatrix}
1 &  & \\
& \lambda_{(V_1, V_2)} &  \\
& & -\lambda_{(V_1, V_2)} \\
\end{bmatrix}.
\]
\item $|\beta_{(V_1, V_2)}| = \lambda_{(V_1, V_2)}$.
\end{enumerate}
\end{theorem}
\begin{proof}
We only need to prove (5). Keeping the foregoing notational convention, we use $\lambda$, $\beta$, and $f$ for $\lambda_{(V_1, V_2)}$, $\beta_{(V_1, V_2)}$, and $f_{(V_1, V_2)}$ respectively. It follows from \eqref{crosscomm} that
\[
V_2^*V_1 f = \beta f.
\]
Since $f$ is a unit vector in $E_1$, we have
\[
\begin{split}
\beta & = \langle \beta f, f \rangle
\\
& = \langle V_2^*V_1 f, f \rangle
\\
& = \langle U^*f, Uf \rangle,
\end{split}
\]
where the last equality follows from the definition of $U$. By parts (a) and (b) of Lemma \ref{Uaction}, we then have
\[
\begin{split}
\beta & = \big\langle \langle U^*f, f_3\rangle f_3, \langle Uf, f_2 \rangle f_2 \big\rangle
\\
& = - \langle U^*f, f_3\rangle \langle f_2, Uf \rangle \lambda,
\end{split}
\]
as $\langle f_3, f_2 \rangle = - \lambda$. The result now follows from the fact that $\langle U^*f, f_3\rangle$ and $\langle f_2, Uf \rangle$ are unimodular constants.
\end{proof}

We record the following particular but useful fact for convenient future retrieval:
\begin{equation}\label{all about 3}
\text{ran}[V_2^*, V_1] = E_1 = \mathbb{C} f_{(V_1, V_2)} = \clw_1 \cap \clw_2,
\end{equation}
where $E_1 = \ker (C(V_1, V_2) - I_{\clh})$. The final equality is a consequence of Lemma \ref{intersection}.

\section{Classification of $3$-finite pairs}\label{sec class of 3 finite}

The purpose of this section is to present a complete classification of irreducible $3$-finite pairs in terms of computable unitary invariants. The structural results of the preceding section will be used thoroughly. Therefore, we continue with $(V_1, V_2)$, an irreducible $3$-finite pair on $\clh$. We adhere to the notational convention established in Section \ref{sec properties of 3 finite} (more specifically, see \eqref{eqn 3 notations 1} and \eqref{eqn 3 notations 2}):
\begin{align*}
&\beta = \beta_{(V_1, V_2)} = \text{unique nonzero eigenvalue of } [V_2^*, V_1],\\
&\lambda = \lambda_{(V_1, V_2)} = \text{unique eigenvalue of } C(V_1, V_2) \text{ in } (0, 1),\\
&f = f_{(V_1, V_2)} = \text{a unit vector in } E_1.	
\end{align*}
Recall that the unitary $U$ of the corresponding BCL triple $(\clw, U, P_{\clw_1})$ is given by (see Theorem \ref{bcl} or \eqref{unitarydef})
\[
U = \begin{bmatrix}
V_2|_{\mathcal{W}_1} & \\
& V_1^*|_{V_1 \mathcal{W}_2}
\end{bmatrix} : \mathcal{W}_1 \oplus V_1 \mathcal{W}_2 \to V_2\mathcal{W}_1 \oplus \mathcal{W}_2.
\]
Also recall from Theorem \ref{alphalambda} (or see \eqref{all about 3}) that
\[
\text{ran}[V_2^*, V_1] = E_1 = \mathbb{C} f = \clw_1 \cap \clw_2,
\]
and $[V_2^*, V_1] f = \beta f$. For all $i = 1, 2, 3, 4$, set
\begin{equation}\label{tilde}
g_i = \frac{1}{\langle f_3, U^*f \rangle} f_i.
\end{equation}
By \eqref{formulauf2} and the fact $\langle f_3, U^*f \rangle$ is a unimodular constant, we have
\begin{equation}\label{finalformula}
\langle U^*f, g_3 \rangle = 1 \text{ and } \langle U^*{g_2}, f \rangle = - \frac{\beta}{\lambda}.
\end{equation}
By Corollary \ref{fi}, we also have
\[
\{g_1, g_3\}, \{g_2, g_4\} \in B_{E^\lambda}.
\]
Further, it follows from Lemma \ref{Uaction} that
\[
U^m f_1 \in \text{ran}Q,
\]
and
\[
U^{*m} f_4 \in \text{ran}Q^\perp,
\]
for all $m \geq 1.$	Set
\[
\widetilde{\mathcal{W}} = \overline{\text{span}} \{f, U^m f_1, U^{*m} f_4 : m \geq 0 \}.
\]
We show that:

\begin{lemma}\label{Etilde}
$\widetilde{\mathcal{W}}$ reduces $(U, P_{\mathcal{W}_1})$.
\end{lemma}
\begin{proof}
Clearly, $f, f_1, f_4 \in \widetilde{\mathcal{W}}$. Recall that $\{e_{\pm \lambda}\} \in B_{E_{\lambda} \oplus E_{-\lambda}}$, and then \eqref{eqn f1 part I} implies
\[
f_1 \pm f_4 = \text{a scalar multiple of } e_{\pm \lambda},
\]
and consequently
\[
E_{\lambda} \oplus E_{-\lambda} = \text{span}\{f_1 \pm f_4\}.
\]
Thus
\[
E_1 \oplus E_\lambda \oplus E_{-\lambda} \subset \widetilde{\mathcal{W}}.
\]
First, we claim that $\widetilde{\mathcal{W}}$ reduces $U$. We know by Lemma \ref{Uaction} that $U f \in \text{span}\{f_2\}$ and $U^* f \in \text{ span} \{f_3\}$, and hence, by Corollary \ref{fi} we conclude that
\[
U f, U^* f \in E_\lambda \oplus E_{-\lambda}.
\]
Therefore, $U f, U^* f \in \widetilde{\mathcal{W}}$. It just remains to show that $U^*f_1, Uf_4 \in \widetilde{\mathcal{W}}$. To this end, observe that Corollary \ref{fi} implies
\[
f_1 \in \text{ran}Q_\lambda \subset E_\lambda \oplus E_{-\lambda} = \text{span} \{f_2, f_4\},
\]
and consequently, it follows from Lemma \ref{Uaction} that
\[
U^*f_1 \in \text{span} \{U^*f_2, U^*f_4 \} = \text{span} \{f, U^* f_4 \} \subset \widetilde{\mathcal{W}}.
\]
Similarly, by Corollary \ref{fi}, we have
\[
f_4 \in E_\lambda \oplus E_{-\lambda} = \text{span} \{f_1, f_3 \}.
\]
By Lemma \ref{Uaction}, we have
\[
Uf_4 \in \text{span} \{U f_1, U f_3\} = \text{span} \{U f_1, f\} \subset \widetilde{\mathcal{W}}.
\]
We have thus proved that $\widetilde{\mathcal{W}}$ reduces $U$.

\noindent We now show that $\widetilde{\mathcal{W}}$ reduces $P_{\mathcal{W}_1}$. As $f \in \mathcal{W}_1$, we have $P_{\mathcal{W}_1}(f) = f$. One easily observes from \eqref{eqn7} and Lemma \ref{Uaction} that
\[
U^m f_1 \in \mathcal{W}_1,
\]
and hence,
\[
P_{\mathcal{W}_1}(U^m f_1) = U^m f_1,
\]
for all $m \geq 0$. By \eqref{eqn7} and Lemma \ref{Uaction}
\[
U^{*m}f_4 \in \text{ran}Q^\perp \subseteq V_1 \mathcal{W}_2,
\]
and consequently
\[
P_{\mathcal{W}_1}(U^{*m} f_4) = 0,
\]
for all $m \geq 1$. Finally, since
\[
f_4 \in E_\lambda \oplus E_{-\lambda} = \text{span} \{f_1, f_3\},
\]
and $f_3 \in V_1 \mathcal{W}_2$ (by Equation \eqref{eqn7}), we have that
\[
P_{\mathcal{W}_1}f_3 = 0,
\]
and therefore
\[
P_{\mathcal{W}_1}f_4 \in \mathbb{C} f_1 \subset \widetilde{\mathcal{W}}.
\]
This shows that $\widetilde{\mathcal{W}}$ reduces $P_{\mathcal{W}_1}$.
\end{proof}

Note that a BCL pair $(V_1, V_2)$ on $\mathcal{H}$ is irreducible if and only if $(U, P_{\mathcal{W}_1})$ is irreducible \cite[Corollary 2.2]{DSPS}. From the above lemma, we obtain:

\begin{coro}\label{onb}
$\{f, f_1, f_3, U^mf_1, U^{*m}f_4: m \geq 1  \} \in B_{\mathcal{W}}$.
\end{coro}
\begin{proof}
As $(V_1, V_2)$ is irreducible, there is no non-trivial closed $(U, P_{\mathcal{W}_1})$-reducing subspace of $\mathcal{W}$ and hence, it follows immediately from the above lemma that
\[
\widetilde{\mathcal{W}} = \mathcal{W}.
\]
Since $f \in E_1$, and
\[
\text{span}\{f_1, f_4\} = E_\lambda \oplus E_{-\lambda} = \text{span}\{f_1, f_3\},
\]
and $U^m f_1 \in \text{ran}Q$, $U^{*m} f_4 \in \text{ran}Q^\perp$ for all $m \geq 1$, it follows that $\{f, f_1, f_3, U^m f_1, U^{*m} f_4 : m \geq 1\}$ is an orthonormal basis for $\mathcal{W}$.
\end{proof}

We are now ready to prove the main result of this section, which serves as a foundational component of the article's overall contribution. Recall (see  Theorem \ref{alphalambda}) again that for an irreducible $3$-finite pair $(V_1, V_2)$,
\[
\beta_{(V_1, V_2)} = \text{unique nonzero eigenvalue of } [V_2^*, V_1].
\]

\begin{theorem}\label{charac}
Let $(V_1, V_2)$ on $\clh$ and $(\tilde{V}_1, \tilde{V}_2)$ on $\tilde{\clh}$ be irreducible $3$-finite pairs. Then $(V_1, V_2)$ and $(\tilde{V}_1, \tilde{V}_2)$ are jointly unitarily equivalent if and only if
\[
\beta_{(V_1, V_2)} = \beta_{(\tilde{V}_1, \tilde{V}_2)}.
\]
\end{theorem}
\begin{proof}
If $(V_1, V_2)$ and $(\tilde{V}_1, \tilde{V}_2)$ are unitarily equivalent, then clearly $\beta_{(V_1, V_2)} = \beta_{(\tilde{V}_1, \tilde{V}_2)}$. For the non-trivial direction, assume that
\[
\beta: = \beta_{(V_1, V_2)} = \beta_{(\tilde{V}_1, \tilde{V}_2)}.
\]
As usual, we define
\[
\clw = \ker (V_1 V_2)^*, \;\clw_i = \ker V_i^*,
\]
and
\[
\tilde{\clw} = \ker (\tilde{V}_1 \tilde{V}_2)^*,\;  \tilde{\clw}_i = \ker \tilde{V}_i^*,
\]
for all $i=1,2$. Define
\[
\lambda:= |\beta|.
\]
By Theorem \ref{alphalambda}, we have
\[
\sigma(C(V_1, V_2)) \cap (0, 1) = \sigma(C(\tilde V_1, \tilde V_2)) \cap (0, 1) = \{\lambda\}.
\]
In view of Theorem \ref{bcl}, we consider the BCL triples $(\clw, U, P_{\clw_1})$ and $(\tilde{\clw}, \tilde{U}, P_{\tilde{\clw}_1})$, where $U$ is the unitary on $\clw$ given by	
\begin{equation*}
U = \begin{bmatrix}
V_2|_{\mathcal{W}_1} & \\
& V_1^*|_{V_1 \mathcal{W}_2}
\end{bmatrix} : \clw_1 \oplus V_1 \mathcal{W}_2 \to V_2 \mathcal{W}_1 \oplus \mathcal{W}_2,
\end{equation*}
and $\tilde{U}$ is the unitary on $\tilde{\clw}$ given by
\begin{equation*}
\tilde{U} = \begin{bmatrix}
\tilde{V}_2|_{\tilde{\clw}_1} & \\
& \tilde{V}_1^*|_{\tilde{V}_1 \tilde{\clw_2}}
\end{bmatrix} : \tilde\clw_1 \oplus \tilde{V_1} \tilde{\mathcal{W}_2} \to \tilde V_2 \tilde{\mathcal{W}_1} \oplus \tilde{\mathcal{W}_2}.
\end{equation*}
It suffices to show that the triples $(\clw, U, P_{\clw_1})$ and $(\tilde{\clw}, \tilde{U}, P_{\tilde{\clw}_1})$ are unitarily equivalent, that is, we claim that there is a unitary $\Pi : \clw \to \tilde{\clw}$ such that
\[
\Pi U \Pi^{-1} = \tilde{U} \text{ and } \Pi P_{\clw_1} \Pi^{-1} = P_{\tilde{\clw}_1}.
\]
We recall that $E_\mu := \ker (C(V_1, V_2) - \mu I_{\clh})$ for a scalar $\mu$ (see \eqref{eqn def eigenspace 2}). Set
\[
\tilde{E}_\mu = \ker (C(\tilde{V}_1, \tilde{V}_2) - \mu I_{\tilde \clh}).
\]
We know (see \eqref{all about 3} or Theorem \ref{alphalambda}) that
\[
\text{ran}[V_2^*, V_1] = E_1 = \mathbb{C} f_{(V_1, V_2)} = \clw_1 \cap \clw_2,
\]
and
\[
\text{ran}[\tilde{V}_2^*, \tilde{V}_1] = \tilde{E}_1 = \mathbb{C} f_{(\tilde{V}_1, \tilde{V}_2)} = \tilde{\clw}_1 \cap \tilde{\clw}_2.
\]
As in the construction of Section \ref{sec properties of 3 finite}, suppose
\[
\{e_{\pm \lambda}\} \in B_{E_{\lambda} \oplus E_{-\lambda}} \text{ and } \{\tilde{e}_{\pm \lambda}\} \in B_{\tilde{E}_\lambda \oplus \tilde{E}_{-\lambda} }.
\]
Also, as in Corollary \ref{fi} (or see \eqref{eqn f1 part I} and \eqref{eqn f1 part II}), define
\[
f_1 = \sqrt{\frac{1+\lambda}{2}} e_{\lambda} + \bar{\alpha}\sqrt{\frac{1-\lambda}{2}} e_{-\lambda},\;  f_4 = \sqrt{\frac{1+\lambda}{2}} e_{\lambda} - \bar{\alpha}\sqrt{\frac{1-\lambda}{2}} e_{-\lambda},
\]
and
\[
f_2 = \frac{f_1 - \lambda f_4}{\sqrt{1 - {\lambda}^2}}, \text{ and }  f_3 = \frac{f_4 - \lambda f_1}{\sqrt{1 - {\lambda}^2}},
\]
for some unimodular constant $\alpha$. Clearly, $f_i \in \clw$, $i=1,2,3,4$. Following \eqref{tilde}, define
\[
g_i:= \frac{1}{\langle f_3, U^* f_{(V_1, V_2)} \rangle} f_i,
\]
for all $i = 1, 2, 3, 4$. By (a) and (b) of Lemma \ref{Uaction} and \eqref{finalformula}, we conclude that
\begin{equation}\label{unitaryfinal_1}
\begin{cases}
U f_{(V_1, V_2)} = -\frac{\bar{\beta}}{\lambda} {g}_2,
\\
U {g}_3 = f_{(V_1, V_2)},
\end{cases}
\end{equation}	
where the first equality follows from the fact that
\[
U f_{(V_1, V_2)} = \langle U f_{(V_1, V_2)}, {g}_2 \rangle {g}_2.
\]
Corollary \ref{onb} yields
\[
\{f_{(V_1, V_2)}, {g}_1, {g}_3, U^m g_1, U^{*m} {g}_4: m \geq 1\} \in B_{\clw}.
\]
Similarly, we have the vectors $\tilde{f}_i \in \tilde{\clw}$, $i=1,2,3,4$, defined by
\[
\tilde{f}_1 = \sqrt{\frac{1+\lambda}{2}} \tilde{e}_{\lambda} + \bar{\tilde{\alpha}}\sqrt{\frac{1-\lambda}{2}} \tilde{e}_{-\lambda},\;  \tilde{f}_4 = \sqrt{\frac{1+\lambda}{2}} \tilde{e}_{\lambda} - \bar{\tilde{\alpha}}\sqrt{\frac{1-\lambda}{2}} \tilde{e}_{-\lambda},
\]
and
\[
\tilde{f}_2 = \frac{\tilde{f}_1 - \lambda \tilde{f}_4}{\sqrt{1 - {\lambda}^2}}, \text{ and }  \tilde{f}_3 = \frac{\tilde{f}_4 - \lambda \tilde{f}_1}{\sqrt{1 - {\lambda}^2}},
\]
for some unimodular constant $\tilde{\alpha}$. Also, define
\[
\tilde{g}_i:= \frac{1}{\langle \tilde{f}_3, \tilde{U}^* f_{(\tilde{V_1}, \tilde{V_2})} \rangle} \tilde{f}_i,
\]
for all $i = 1, 2, 3, 4$. As above, we again have
\begin{equation}\label{unitaryfinal_2}
\begin{cases}
\tilde{U} f_{(\tilde{V_1}, \tilde{V_2})} = -\frac{\bar{\beta}}{\lambda} \tilde{g}_2
\\
\tilde{U} \tilde{g}_3 = f_{(\tilde{V_1}, \tilde{V_2})},
\end{cases}
\end{equation}			
and
\[
\{f_{(\tilde{V_1}, \tilde{V_2})}, \tilde{g}_1, \tilde{g}_3, \tilde{U}^m \tilde{g}_1, \tilde{U}^{*m} \tilde{g}_4: m \geq 1\} \in B_{\tilde \clw}.
\]
A simple computation shows that
\begin{equation}\label{tilde4}
\begin{cases}
{g}_4 = (\sqrt{1 - \lambda^2}) {g}_3 + \lambda {g}_1
\\
\tilde{g}_4 = (\sqrt{1 - \lambda^2}) \tilde{g}_3 + \lambda \tilde{g}_1.
\end{cases}
\end{equation}	
Define the unitary $\Pi: \clw \to \tilde{\clw}$ by
\[
\Pi f_{(V_1, V_2)} = f_{(\tilde{V_1}, \tilde{V_2})},
\]
and
\[
\Pi f =
\begin{cases}
\tilde{g}_j & \mbox{if }f = g_j \text{ for } j=1,3
\\
\tilde{U}^m \tilde{g}_1 & \mbox{if } f = U^m g_1
\\
\tilde{U}^{*m} \tilde{g}_4 & \mbox{if } f = {U}^{*m} {g}_4,
\end{cases}
\]
for $m \geq 1$. It is clear from the definition of $\Pi$ and the identity \eqref{tilde4} that
\[
\Pi{g}_4 = \tilde{g}_4,
\]
and consequently
\[
\begin{split}
\Pi{g}_2 = \Pi \Big(\frac{{g}_1 - \lambda {g}_4}{\sqrt{1 - \lambda^2}} \Big) = \frac{\tilde{g}_1 - \lambda \tilde{g}_4}{\sqrt{1 - \lambda^2}},
\end{split}
\]
that is,
\begin{equation}\label{tilde2}
\Pi {g}_2 = \tilde{g}_2.
\end{equation}
It is now easy to see that
\[
(\Pi U \Pi^{-1})\tilde{g}_1 = \tilde{U} \tilde{g}_1.
\]
Moreover
\[
\begin{split}
(\Pi U \Pi^{-1})\tilde{g}_3 = \Pi U g_3
= \Pi f_{(V_1, V_2)} = f_{(\tilde{V_1}, \tilde{V_2})} = \tilde{U} \tilde{g_3},
\end{split}
\]
where the second and fourth equalities follow by an appeal to \eqref{unitaryfinal_1} and \eqref{unitaryfinal_2} respectively.
For all $m \geq 1$, we also have
\[
\begin{split}
(\Pi U \Pi^{-1})\big(\tilde{U}^m \tilde{g}_1\big) & = (\Pi U)\big(U^m g_1 \big) = \Pi \big(U^{m+1} g_1\big) = (\tilde{U})^{m+1} \tilde{g}_1 = \tilde{U} \big(\tilde{U}^{m} \tilde{g}_1 \big),
\end{split}
\]
and similarly,
\[
(\Pi U \Pi^{-1})\big(\tilde{U}^{*m} \tilde{g}_4\big) = \tilde{U}^{*(m-1)} \tilde{g}_4 = \tilde{U}(\tilde{U}^{*m} \tilde{g}_4).
\]
Finally, by \eqref{unitaryfinal_1}, \eqref{tilde2}, and \eqref{unitaryfinal_2}, it follows that
\[
\begin{split}
(\Pi U \Pi^{-1}) f_{(\tilde{V_1}, \tilde{V_2})} = \Pi U f_{(V_1, V_2)} &= \Pi \big(-\frac{\bar{\beta}}{\lambda} {g}_2 \big) = -\frac{\bar{\beta}}{\lambda}\tilde{g}_2 =  \tilde{U}f_{(\tilde{V_1}, \tilde{V_2})}.
\end{split}
\]
This proves that $\Pi U \Pi^{-1} = \tilde{U}.$ The verification of $\Pi P_{\clw_1} \Pi^{-1} = P_{\tilde{\clw}_1}$ is easy and is left to the reader. This completes the proof of the theorem.
\end{proof}

Now it is important to furnish explicit examples of irreducible $3$-finite pairs. We start with invariant subspaces of $L^2(\T^2)$. For each nonzero $r$ in $(-1, 1)$ and unimodular function $\varphi \in L^\infty(\T^2)$, define
\[
\cll_{\vp, r} = \varphi \Big(H^2(\mathbb{D}^2) \bigoplus \Big(\bigoplus_{j = 0}^\infty z^j \text{span} \Big\{ \frac{\bar{w}}{1 - r z \bar{w}}\Big\} \Big)\Big)
\]
Then $\mathcal{L}_{\varphi, r}$ is (jointly) invariant under $(L_z, L_w)$, and
\[
[(L_w|_{\cll_{\vp, r}})^*, L_z|_{\cll_{\vp, r}}] = [(L_w|_{\cll_{\vp, r}})^*, L_z|_{\cll_{\vp, r}}]^* \neq 0.
\]
In fact, Izuchi and Ohno proved that $\mathcal{L}_{\varphi, r}$ are the only invariant subspaces of $L^2(\T^2)$ that satisfies the above self-adjoint condition (see \cite[Theorem 1 ]{IO94}). This observation was one of the keys to the construction of invariant subspaces of $H^2(\D^2)$ with self-adjoint and nonzero cross-commutators \cite{II2006}.

\begin{example}\label{bidiskex}
For each  nonzero $r$ in $(-1, 1)$, there exists an inner function $\vp \in H^\infty(\mathbb{D}^2)$ such that \cite[Theorem 2]{II2006}
\begin{equation} \label{submodule}
\cls_{r} = \vp \Big(H^2(\mathbb{D}^2) \bigoplus \Big(\bigoplus_{j = 0}^\infty z^j \text{span} \Big\{ \frac{\bar{w}}{1 - r z \bar{w}}\Big\} \Big) \Big).
\end{equation}
is an invariant subspace of $H^2(\mathbb{D}^2)$ (see \eqref{eqn: inv sub Hardy} for the definition of invariant subspaces of $H^2(\D^2)$) and
\[
[(M_w|_{\cls_{r}})^*, M_z|_{\cls_{r}}]^* = [(M_w|_{\cls_{r}})^*, M_z|_{\cls_{r}}] \neq 0,
\]
and
\[
\text{rank}[(M_w|_{\cls_{r}})^*, M_z|_{\cls_{r}}] = 1.
\]
A simple computation reveals that (see \cite[proof of Theorem 3]{II2006}) $r$ is the only nonzero eigenvalue of $[(M_w|_{\cls_{r}})^*, M_z|_{\cls_{r}}]$. Moreover, the pair $(M_z|_{\cls_{r}}, M_w|_{\cls_{r}})$ is an irreducible $3$-finite pair (See \cite[Example 8.9, Page 246]{Yang-S}, and
\[
\sigma(C(M_z|_{\cls_{r}}, M_w|_{\cls_{r}})) \cap (0,1) = \{|r|\}.
\]
That is, $|r|$ is the unique eigenvalue of the defect operator of $(M_z|_{\cls_{r}}, M_w|_{\cls_{r}})$ lying in $(0, 1)$.
\end{example}

\begin{example}\label{bidiskex_1}
Let $\gamma$ be a unimodular constant and let $r$ be a nonzero real number in $(-1, 1)$. Consider the submodule $\cls_{r}$ as in \eqref{submodule}, and then, consider the isometric pair $(\gamma M_z|_{\cls_{r}}, M_w|_{\cls_{r}})$ on $\cls_{r}$. It is easy to see that
\[
C\big(\gamma M_z|_{\cls_{r}}, M_w|_{\cls_{r}}\big) = C(M_z|_{\cls_{r}}, M_w|_{\cls_{r}}),
\]
and
\[
[(M_w|_{\cls_{r}})^*, \gamma M_z|_{\cls_{r}}] = \gamma [(M_w|_{\cls_{r}})^*, M_z|_{\cls_{r}}].
\]
It follows immediately from the discussion of Example \ref{bidiskex} that $(\gamma M_z|_{\cls_{r}}, M_w|_{\cls_{r}})$ is an irreducible $3$-finite pair with
\[
\sigma(C(\gamma M_z|_{\cls_{r}}, M_w|_{\cls_{r}})) \cap (0,1) = \{|r|\},
\]
and
\[
\sigma([(M_w|_{\cls_{r}})^*, \gamma M_z|_{\cls_{r}}]) \cap (0,1) = \{\gamma r\}.
\]
\end{example}

The above example, together with Theorem \ref{charac}, immediately yields a complete characterization of irreducible 3-finite pairs.

\begin{theorem}\label{irrfinal}
Let $(V_1, V_2)$ be an irreducible $3$-finite pair on a Hilbert space $\mathcal{H}$. Then there exist $\lambda \in (0, 1)$, a unimodular constant $\gamma$, and an inner function $\vp \in H^\infty(\D^2)$ (depending on $\lambda$) such that
\[
\sigma(C(V_1, V_2)) \cap (0,1)  = \{\lambda\},
\]
and
\[
\sigma([V_2^*, V_1]) \setminus \{0\} = \{\gamma \lambda\},
\]
and
\[
(V_1, V_2) \cong (\gamma M_z|_{\cls_{\lambda}}, M_w|_{\cls_{\lambda}})
\]
where $\cls_{\lambda}$ is the invariant subspace of $H^2(\mathbb{D}^2)$ of the form
\[
\cls_{\lambda} = \vp \Big(H^2(\mathbb{D}^2) \bigoplus \Big(\bigoplus_{j = 0}^\infty z^j \text{span} \Big\{ \frac{\bar{w}}{1 - \lambda z \bar{w}}\Big\} \Big) \Big).
\]

\end{theorem}
\begin{proof}
The existence of $\lambda$ and $\gamma$ follow from Theorem \ref{alphalambda}. Here, note that
\[
[V_2^*, V_1] f = \beta f,
\]
and $\gamma$ is the unique unimodular constant such that $\beta = \gamma \lambda$ (compare with $(3)$ and $(5)$ of Theorem \ref{alphalambda}). With this particular $\lambda$ and $\gamma$, we now apply Example \ref{bidiskex_1} to conclude that
\[
(V_1, V_2) \cong (\gamma M_z|_{\cls_{\lambda}}, M_w|_{\cls_{\lambda}}),
\]
for some inner function $\vp \in H^\infty(\D^2)$. This completes the proof of the theorem.
\end{proof}

The following intriguing and direct ramifications of Theorem \ref{charac} and Example \ref{bidiskex_1} are worth highlighting: Let $\gamma_1$ and $\gamma_2$ be unimodular constants, and let $r$ be a nonzero number in $(-1, 1)$. Consider the invariant subspace $\cls_{r}$ of $H^2(\mathbb{D}^2)$ as defined in \eqref{submodule}. Then:
\begin{enumerate}
\item $(\gamma_1 M_z|_{\cls_{\lambda}}, M_w|_{\cls_{\lambda}}) \cong (\gamma_2 M_z|_{\cls_{\lambda}}, M_w|_{\cls_{\lambda}})$ if and only if $\gamma_1 = \gamma_2$.
\item $(\gamma_1 M_z|_{\cls_{\lambda}}, M_w|_{\cls_{\lambda}}) \cong (M_z|_{\cls_{\lambda}}, \gamma_2 M_w|_{\cls_{\lambda}})$ if and only if $\gamma_1 = \bar{\gamma}_2$.
\end{enumerate}

\section{Classification of $2$-finite pairs}\label{sec 2 finite}

The focus of this section is a complete characterization of irreducible $2$-finite pairs. Let $(V_1, V_2)$ be an irreducible $2$-finite pair on $\mathcal{H}$. First, we claim the following crucial spectral property:
\[
\sigma(C(V_1, V_2)) \setminus \{0\} = \{\pm 1\}.
\]
Indeed, an argument similar to the proof of Proposition \ref{dime1} yields that, in this case also
\[
\dim E_1 = 1.
\]
Since $\text{rank} C(V_1, V_2) = 2$, it follows immediately from Theorem \ref{structure} that
\[
\dim E_{-1} = 1.
\]
We conclude that the only nonzero eigenvalues of $C(V_1, V_2)$ are $\{\pm 1\}$, and
\[
\dim E_1 = 1 = \dim E_{-1}.
\]
Therefore
\[
(\ker C(V_1, V_2))^\perp = E_1 \oplus E_{-1}.
\]
Let $e_1$ and $e_{-1}$ denote unit vectors in $E_1$ and $E_{-1}$ respectively. Then Corollary \ref{cor rank formula} and Lemmas \ref{intersection} and \ref{contain} imply that
\[
\text{rank}[V_2^*, V_1] = \dim E_1 = 1,
\]
and
\[
\text{ran}[V_2^*, V_1] = E_1,
\]
and
\[
[V_2^*, V_1]|_{E_1^\perp} = 0.
\]
In particular, there exists a nonzero scalar $\alpha$ such that
\[
[V_2^*, V_1] = \alpha e_1 \otimes e_1.
\]
Same computation as in Section \ref{sec properties of 3 finite} (see, in particular, the equality \eqref{crosscomm}) yields that
\[
V_2^* V_1 e_1 = \alpha e_1, \text{ and } V_1^* V_2e_1 = \overline{\alpha}e_1.
\]
Let
\[
\mathcal{N} = \mathcal{W} \ominus (\ker C(V_1, V_2))^\perp.
\]
Clearly, with respect to the decomposition $\mathcal{W} = E_1 \oplus E_{-1} \oplus \mathcal{N}$, we have
\[
C(V_1, V_2)|_{\mathcal{W}} = \begin{bmatrix}
I_{E_1} &  &  \\
& -I_{E_{-1}} &  \\
& & 0
\end{bmatrix}.
\]
Recall that $C(V_1, V_2)|_{\mathcal{W}} = P_{\mathcal{W}_1}- P_{V_2 \mathcal{W}_1}$. Hence
\[
P_{\mathcal{W}_1}- P_{V_2 \mathcal{W}_1} = \begin{bmatrix}
I_{E_1} &  &  \\
& -I_{E_{-1}} &  \\
& & 0
\end{bmatrix}.
\]
Thus, Theorem \ref{struc} implies
\[
P_{\mathcal{W}_1} = \begin{bmatrix}
	I_{E_1} &  &  \\
	& 0 &  \\
	& & Q
\end{bmatrix},
\]
and
\[
P_{V_2\mathcal{W}_1} = \begin{bmatrix}
0 &  &  \\
& I_{E_{-1}} &  \\
& & Q
\end{bmatrix},
\]
where $Q : \mathcal{N} \to \mathcal{N}$ is a projection. Then
\[
P_{\mathcal{W}_2} = I_{\mathcal{W}} - P_{V_2 \mathcal{W}_1} = \begin{bmatrix}
I_{E_1} &  &  \\
& 0 &  \\
& & Q^\perp
\end{bmatrix},
\]
and
\[
P_{V_1\mathcal{W}_2} = I_{\mathcal{W}} - P_{\mathcal{W}_1} = \begin{bmatrix}
0 &  &  \\
& I_{E_{-1}} &  \\
& & Q^\perp
\end{bmatrix}.
\]
Consider the unitary $U$ on $\mathcal{W}$ as given by Theorem \ref{bcl}:
\[
\clw = \clw_1 \oplus V_1 \clw_2 = \clw_2 \oplus V_2 \clw_1,
\]
and
\[
U = \begin{bmatrix}
V_2|_{\clw_1} & \\
& V_1^*|_{V_1 \clw_2}
\end{bmatrix} : \clw_1 \oplus V_1 \clw_2 \to V_2\clw_1 \oplus \clw_2.
\]
Since
\[
\mathcal{W}_1 = E_1 \oplus \text{ran}Q,
\]
and
\[
V_2{\mathcal{W}_1} = E_{-1} \oplus \text{ran}Q,
\]
and since $U = V_2$ on $\mathcal{W}_1$, it follows that
\[
V_2e_1 = Ue_1 = \langle Ue_1, e_{-1} \rangle e_{-1} + g,
\]
for some $g \in \text{ran}Q$. Since $\text{ran}Q \subset \mathcal{W}_1$, applying $V_1^*$ on both sides of the preceding equation and then using the definition of $U$, we obtain that
\[
\begin{split}
\overline{\alpha} e_1 & = V_1^* V_2 e_1
\\
& = \langle Ue_1, e_{-1} \rangle V_1^*e_{-1}
\\
& = \langle Ue_1, e_{-1} \rangle U e_{-1}.
\end{split}
\]
This shows that
\[
U(E_{-1}) = E_1.
\]
Moreover, since $\langle Ue_1, e_{-1}\rangle$ is a unimodular constant, we have
\[
|\alpha| = 1.
\]
As
\[
\mathcal{W}_2 = E_1 \oplus \text{ran}Q^\perp, \;  V_1{\mathcal{W}_2} = E_{-1} \oplus \text{ran}Q^\perp,
\]
and $U^* = V_1$ on $\mathcal{W}_2$, there exists $h \in \text{ran}Q^\perp$ such that
\[
V_1e_{1} = U^*e_{1} = \langle U^* e_{1}, e_{-1} \rangle e_{-1} + h.
\]
As before, since $\text{ran}Q^\perp \subset \mathcal{W}_2$, by applying $V_2^*$ on both sides of the preceding equality and then using the definition of $U$ we find that
\[
\begin{split}
\alpha e_1 & = V_2^* V_1 e_1
\\
& = \langle U^* e_1, e_{-1} \rangle V_2^*e_{-1}
\\
& = \langle U^*e_1, e_{-1} \rangle U^* e_{-1}.
\end{split}
\]
This shows that
\[
U^*(E_{-1}) = E_1,
\]
and hence, $E_1 \oplus E_{-1}$ reduces $U$. On the other hand, we know that $E_1 \subset \mathcal{W}_1$ and $E_{-1} (\subset V_1 \mathcal{W}_2)$ is orthogonal to $\mathcal{W}_1$. It is now obvious that $E_1 \oplus E_{-1}$ reduces $P_{\mathcal{W}_1}$. In other words, $E_1 \oplus E_{-1} \subseteq \clw$ reduces $(U, P_{\mathcal{W}_1})$. But $(V_1, V_2)$ and equivalently $(U, P_{\mathcal{W}_1})$ is irreducible. Therefore,
\[
\mathcal{W} = E_1 \oplus E_{-1}.
\]
Then $\mathcal{N} = \{0\}$ and hence
\[
\mathcal{W}_1 = E_1 = \mathcal{W}_2.
\]
At this point, we recall the following result from Bercovici, Douglas, and Foias \cite[Corollary 4.3]{BDF}:

\begin{theorem}\label{shift}
Let $(T_1, T_2)$ be an irreducbible isometric pair. Suppose
\[
\dim(\ker T_i^*) < \infty \qquad (i=1,2).
\]
Then each $T_i$, $i=1,2$, is either shift, or a constant multiple of the identity.
\end{theorem}

Returning to our context, we immediately have the following:

\begin{coro}
$V_1$ and $V_2$ are unilateral shifts.
\end{coro}

Since
\[
\mathcal{W}_1 = \ker V_1^* = \text{span}\{e_1\},
\]
we conclude that $V_1$ is a unilateral shift of multiplicity one, that is, $V_1 \cong M_z$ on $H^2(\mathbb{D})$. More specifically
\[
W(V_1^n e_1) = z^n \qquad (n \geq 0),
\]
defines a unitary $W : \mathcal{H} \to H^2(\mathbb{D})$ such that
\[
WV_1 = M_z W.
\]
Then $(V_1, V_2)$ on $\mathcal{H}$ is jointly unitarily equivalent to $(M_z, WV_2W^*)$ on $H^2(\mathbb{D})$. As $WV_2W^*$ commutes with $M_z$, there exists an inner function $\theta \in H^\infty(\mathbb{D})$ such that
\[
WV_2W^* = M_{\theta}.
\]
Again it follows from
\[
\begin{split}
\dim (\ker M_{\theta}^*) & = \dim (\ker WV_2^*W^*)
\\
& = \dim (\ker V_2^*)
\\
& = \dim \mathcal{W}_2
\\
&= 1,
\end{split}
\]
that
\[
\theta(z) = c \frac{z-a}{1 - \bar{a}z} \qquad (z \in \mathbb{D}),
\]
for some $a \in \mathbb{D}$ and unimodular constant $c$. Consequently,
\[
\ker M_\theta^* = \text{span}\{k_a\},
\]
where
\[
k_a(z) = \frac{1}{1 - \bar{a}z} \qquad (z \in \mathbb{D}),
\]
is the \textit{Szeg\"{o} kernel} on $\D$. But
\[
\mathcal{W}_1 = \mathcal{W}_2,
\]
implies that
\[
\mathbb{C} = \ker M_z^* = \ker M_\theta^* = \text{span } \{k_a\},
\]
which forces that $a = 0$. We have thus proved that:
\[
(V_1, V_2) \text{ on } \mathcal{H} \cong (M_z, cM_z) \text{ on } H^2(\mathbb{D}),
\]
for some  unimodular constant $c$. Consequently
\begin{align*}
	\alpha &= \text{ nonzero eigenvalue of the cross-commutator $[V_2^*, V_1]$ on $\mathcal{H}$ } \\
	&= \text{ nonzero eigenvalue of the cross-commutator $[(cM_z)^*, M_z] $ on $H^2(\mathbb{D})$}\\
	&= \bar{c}.
\end{align*}

The summary of the above observations provides a complete classification of irreducible $2$-finite pairs:

\begin{theorem}\label{rank2}
Let $(V_1, V_2)$ be an irreducible $2$-finite pair. Then $\{\pm 1\}$ are the only nonzero eigenvalues of $C(V_1, V_2)$. Moreover
\[
\text{rank}[V_2^*,V_1] = 1,
\]
and there exists a unimodular constant $\alpha$ such that
\[
\sigma([V_2^*,V_1])\setminus \{0\} = \{\alpha\}.
\]
Moreover
\[
(V_1, V_2) \cong (M_z, \bar{\alpha}M_z).
\]
Conversely, if $\alpha$ is unimodular constant, then $(M_z, \bar{\alpha}M_z)$ on $H^2(\mathbb{D})$ is an irreducible $2$-finite pair with $\{\pm 1\}$ as the only nonzero eigenvalues of $C(M_z, \bar{\alpha}M_z)$.
\end{theorem}

Note that the pair $(M_z, \bar{\alpha}M_z)$ is acting on the Hardy space $H^2(\mathbb{D})$. The details of the converse part of the above result are routine, and we leave the details to the reader.

The following result, which is an immediate consequence of Theorem \ref{rank2}, says that for a $2$-finite pair, the nonzero eigenvalue of the cross-commutator is a complete invariant.

\begin{theorem}\label{thm 2 finite unit inv}
Let $(V_1, V_2)$ on $\clh$ and $(\tilde{V_1}, \tilde{V_2})$ on $\tilde{\clh}$ be irreducible $2$-finite pairs. Let
\[
\sigma([V_2^*,V_1])\setminus \{0\} = \{\alpha\},
\]
and
\[
\sigma([\tilde{V_2}^*,\tilde{V_1}])\setminus \{0\} = \{\tilde{\alpha}\}.
\]
Then \[(V_1, V_2) \cong (\tilde{V_1}, \tilde{V_2}),
\]
if and only if
\[
\alpha = \tilde{\alpha}.
\]
\end{theorem}

\section{Classification of $1$-finite pairs}\label{sec 1 finite}

This short section classifies irreducible $1$-finite pairs. In contrast to $3$ and $2$-finite pairs, this class is simple, and the structure can be easily derived. Indeed, this is simply the pair of shifts on $H^2(\D^2)$:

\begin{theorem}\label{thm 1 finite}
Let $(V_1, V_2)$ be an irreducible isometric pair on a Hilbert space $\clh$. Then $(V_1, V_2)$ is $1$-finite if and only if
\[
(V_1, V_2) \cong (M_z, M_w) \text{ on } H^2(\D^2).
\]
\end{theorem}
\begin{proof}
It is a standard fact that $(M_z, M_w)$ is irreducible. Moreover, $(M_z, M_w)$ is doubly commuting, that is, $[M_w^*,M_z] = 0$, and also (see the identity (3) preceding Definition \ref{def defect operator})
\[
C(M_z, M_w) = P_{\mathbb{C}}.
\]
Therefore, $(M_z, M_w)$ on $H^2(\D^2)$ is an irreducible $1$-finite pair. For the reverse direction, consider an $1$-finite irreducible isometric pair $(V_1, V_2)$ on a Hilbert space $\clh$. We again recall that
\begin{enumerate}
\item $E_1 = \ker V_1^* \cap \ker V_2^*$ (see Lemma \ref{intersection}),
\item $[V_2^*,V_1]|_{E_1^\perp} = 0$, and $\text{ran} [V_2^*,V_1] \subseteq E_1$ (see Lemma \ref{contain}), and
\item $\text{rank} C(V_1, V_2) = 2 \text{rank} [V_2^*,V_1] + \dim E_1 - \dim E_{-1}$ (see Corollary \ref{cor rank formula}).
\end{enumerate}

\noindent By assumption, $\text{rank} C(V_1, V_2) = 1$. Then, in view of Theorem \ref{structure}, we know that either $1$ or $-1$ is the only nonzero eigenvalue of $C(V_1, V_2)$. If $-1$ is the only nonzero eigenvalue of $C(V_1, V_2)$, then
\[
\dim E_1 = 0.
\]
Therefore
\[
\text{rank} [V_2^*, V_1] = 0,
\]
and hence, by the rank identity (3) above, we have
\[
\begin{split}
\text{rank} C(V_1, V_2) & = 2 \times 0 + 0 - 1
\\
& = -1,
\end{split}
\]
an impossibility. Therefore, $1$ is the only nonzero eigenvalue of $C(V_1, V_2)$. Then
\[
E_{-1} = \{0\}.
\]
Since $\text{rank} C(V_1, V_2) = 1$, the rank identity in (3) above again forces that
\[
[V_2^*, V_1] = 0,
\]
that is, $(V_1, V_2)$ is a doubly commuting pair on $\clh$. The Wold decomposition of doubly commuting pairs \eqref{eqn Wold DC} yields the orthogonal decomposition into reducing subspaces
\[
\clh = \clh_{uu} \oplus \clh_{us} \oplus \clh_{su} \oplus \clh_{ss},
\]
where $V_1|_{\clh_{ij}}$ is a shift if $i=s$ and unitary if $i=u$, and $V_2|_{\clh_{ij}}$ is a shift if $j=s$ and unitary if $j=u$. Nevertheless, due to the irreducibility of $(V_1, V_2)$, precisely one summand will survive. We claim that $\clh_{ss}$ is the one who will last. Indeed, if $(W_1, W_2)$ is an isometric pair such that at least one of $W_1$ and $W_2$ is unitary, then an easy computation reveals that
\[
C(W_1, W_2) = 0.
\]
Consequently, in the present situation, we have that
\[
C(V_1|_{\clh_{ij}}, V_2|_{\clh_{ij}}) = 0,
\]
whenever at least one of $i,j$ is $u$. Therefore
\[
\clh = \clh_{ss} \neq \{0\}.
\]
The representation of shift part of the Wold decomposition for doubly commuting pairs \cite[Theorem 3.1]{JS} yields
\[
\clh_{ss} = \bigoplus_{m,n \geq 0} V_1^m V_2^n \Big(\ker (V_1|_{\clh_{ss}})^* \cap \ker (V_2|_{\clh_{ss}})^*\Big).
\]
However, we know from (1) above that
\[
E_1 = \ker (V_1|_{\clh_{ss}})^* \cap \ker (V_2|_{\clh_{ss}})^*.
\]
Since $\dim E_1 = 1$, there exists unit vector $f \in \clh$ such that
\[
E_1 = \mathbb{C} f.
\]
Therefore, there exists a unitary $U: \clh_{ss} \rightarrow H^2(\D^2)$ such that
\[
U(V_1^m V_2^n f) = z^m w^n \qquad (m,n \geq 0).
\]
Moreover, $U V_1 = M_z U$ and $U V_2 = M_w U$ (see \cite{JS} for more details), that is, $(V_1, V_2) \cong (M_z, M_w)$. This completes the proof of the theorem.
\end{proof}

With this, we now have a thorough understanding of irreducible $n$-finite pairings for all $n = 1,2,3$. In the following section aims to show that irreducible $1, 2$, and $3$-finite pairs are all irreducible $n$-finite pairs.

\section{Compact normal pairs}\label{sec class n finite}

In this section, we obtain complete representations of compact normal pairs. As we will see, aggregating all previously learned results will archive this. Indeed, we will see that along with the $3$ and $2$ and $1$-finite pairs obtained before, shift-unitary pairs (see Definition \ref{def doubly comm}) will also serve as the fundamental building blocks of compact normal pairs.

We fix a compact normal pair $(V_1, V_2)$ on $\clh$. As usual, following \eqref{eqn def eigenspace 2}, we write
\[
E_\lambda:= E_\lambda(C(V_1, V_2)) \qquad (\lambda \in \mathbb{R}).
\]
Recall from Lemmas \ref{intersection} and \ref{contain} that
\[
\text{ran}[V_2^*, V_1] = \text{ran}[V_2^*, V_1]^* \subseteq E_1 = \clw_1 \cap \clw_2,
\]
where $\clw_i = \ker V_i^*$, $i=1,2$, and
\[
[V_2^*, V_1]|_{E_1^\perp} = [V_1^*, V_2]|_{E_1^\perp} = 0.
\]
We first consider the case when $\dim E_1 \geq 1$. The case when $E_1 = \{0\}$ is easy and will be discussed later in Remark \ref{e1zero}. Let
\[
\dim E_1 := k \in \mathbb{N} \cup \{\infty\},
\]
and suppose $\{f_1, \ldots, f_k\}$ is an orthonormal basis of $E_1$ consisting of eigen vectors of $[V_2^*, V_1]$ (by treating $[V_2^*, V_1]|_{E_1}$ on $E_1$ as a normal operator). There exist scalars $\{\lambda_1, \ldots,\lambda_k\}$  (possibly repeated) such that
\[
[V_2^*, V_1]f_i = \lambda_i f_i,
\]
for $i = 1, 2, \ldots, k$. Finally, for each $i=1, \ldots, k$, we define a closed subspace $\clh_i$ of $\clh$ as
\[
\mathcal{H}_i := \overline{\text{span}}\{V_1^m V_2^n f_i: m, n \geq 0\}.
\]
These spaces are of interest, which we now analyze thoroughly. First, we prove that these spaces are jointly reducing (see Definition \ref{def irred}).

\begin{lemma}\label{Hi}
$\mathcal{H}_i$ reduces $(V_1, V_2)$ for all $i = 1, 2, \ldots, k$.
 \end{lemma}
\begin{proof}
The proof is exactly the same as the proof of the reducibility of $\cls$ in Proposition \ref{dime1}.
\end{proof}

Moreover, we claim that:

\begin{lemma}\label{irreducuble}
$(V_1|_{\mathcal{H}_i}, V_2|_{\mathcal{H}_i})$ is irreducible for all $i = 1, 2, \ldots, k$.
\end{lemma}
\begin{proof}
Fix an $i$. Note that
\[
C(V_1|_{\mathcal{H}_i}, V_2|_{\mathcal{H}_i}) = C(V_1, V_2)|_{\mathcal{H}_i}.
\]
By the definition of $\mathcal{H}_i$, we have
\[
E_1\big(C(V_1|_{\mathcal{H}_i}, V_2|_{\mathcal{H}_i})\big) = \mathbb{C} f_i.
\]
Suppose $\mathcal{K}$ is a nonzero closed subspace of $\clh_i$. Assume that $\clk$ reduces $(V_1, V_2)$. Since
\[
\dim E_1\big(C(V_1|_{\mathcal{H}_i}, V_2|_{\mathcal{H}_i})\big) = 1,
\]
an easy consequence of the spectral theorem for compact and self-adjoint operators (cf. \cite[Lemma 2.6]{DSPS}) implies that $f_i \in \mathcal{K}$, and consequently, $\mathcal{K} = \mathcal{H}_i$. This completes the proof.
\end{proof}

The following orthogonality relation will be useful in what follows.

\begin{lemma}\label{lemma inner prod 0}
$\big\langle V_2^mf_i, V_1^nf_j \big \rangle = 0$ for $i \neq j$ and $m, n \in \Z_+$.
\end{lemma}
\begin{proof}
Since $f_j \in \clw_1 \cap \clw_2$, it follows that
\[
\begin{split}
V_2^*V_1 f_j & = [V_2^*, V_1]f_j + V_1 V_2^* f_j
\\
& = \lambda_j f_j.
\end{split}
\]
Then, similar computation as in the proof of Proposition \ref{dime1} (or, see \eqref{impeqn}) yields that (note that $V_2^* V_1^n f_j  = V_1^{n-1}V_2^*V_1f_j$)
\[
V_2^* V_1^n f_j = V_1^{n-1}V_2^*V_1f_j = \lambda_j V_1^{n-1}f_j,
\]
for all $n \geq 1$. Repeated application of the above yields
\[
V_2^{*m} V_1^n f_j =
\begin{cases}
\lambda_j^m V_1^{n-m}f_j & \mbox{if } m \leq n\\
0 & \mbox{if } m > n,
\end{cases}
\]
where the final equality is due to the fact that $V_2^{*(m-n)} f_j = 0$ for $m>n$. Then the above equality implies
\[
\begin{split}
\big\langle V_2^mf_i, V_1^nf_j \big \rangle & = \big\langle f_i, V_2^{*m}V_1^nf_j \big \rangle
\\
& = \begin{cases}
\big\langle f_i, \lambda_j^m V_1^{n-m} f_j \big \rangle & \text{ if } m \leq n,\\
0 & \text{ if } m > n.
\end{cases}
\\
& = 0,
\end{split}
\]
and completes the proof of the lemma.
\end{proof}

It is now natural to expect that:

\begin{lemma}\label{ortho}	
$\mathcal{H}_i \perp \mathcal{H}_j$ for all $i \neq j$.
\end{lemma}
\begin{proof}
Suppose $i \neq j$. It is enough to show that
\[
\{V_1^m V_2^nf_i : m, n \geq 0\} \perp \{V_1^m V_2^n f_j : m, n \geq 0\}.
\]
Let $m_1, n_1, m_2, n_2 \geq 0$. Since $f_i, f_j \in \mathcal{W}_1 \cap \mathcal{W}_2$, it follows that
\[
\big\langle V_1^{m_1}V_2^{n_1}f_i, V_1^{m_2}V_2^{n_2}f_j \big\rangle = 0,
\]
whenever $m_1 \geq m_2, n_1 \geq n_2$ or $m_1 \leq m_2, n_1 \leq n_2$. If $m_1 < m_2$ and $n_1 > n_2$, then Lemma \ref{lemma inner prod 0} implies
\[
\begin{split}
\big\langle V_1^{m_1}V_2^{n_1}f_i, V_1^{m_2}V_2^{n_2}f_j \big\rangle & = \big\langle V_2^{n_1-n_2}f_i, V_1^{m_2-m_1}f_j \big\rangle
\\
& = 0.
\end{split}
\]
The remaining case when $m_1 > m_2$ and $n_1 < n_2$ is treated in a similar manner and is left to the reader.
\end{proof}

The nonzero part of the defect operator is contained in the direct sum of $\clh_i$'s. More specifically:

\begin{lemma}\label{kernel}
$\big(\ker C(V_1, V_2)\big)^\perp \subseteq \oplus_{i = 1}^k \mathcal{H}_i.$
\end{lemma}
\begin{proof}
Let
\[
\clh_0 = \mathcal{H} \ominus \big(\oplus_{i = 1}^k \mathcal{H}_i\big).
\]
Then $\clh_0$ reduces $(V_1, V_2)$, and consequently, $(V_1|_{\clh_0}, V_2|_{\clh_0})$ is a BCL pair on $\clh_0$. Clearly
\[
C(V_1|_{\clh_0}, V_2|_{\clh_0}) = C(V_1, V_2)|_{\clh_0},
\]
and
\[
[(V_2|_{\clh_0})^*, V_1|_{\clh_0}] = [V_2^*, V_1]|_{\clh_0}.
\]
By the definition of $\mathcal{H}_i$'s, we have
\[
E_1 \subset \oplus_{i = 1}^k \mathcal{H}_i.
\]
Since $[V_2^*, V_1] = 0$ on $E_1^\perp$, we conclude that
\[
[(V_2|_{\clh_0})^*, V_1|_{\clh_0}] = 0,
\]
and
\[
\dim E_1(C(V_1|_{\clh_0}, V_2|_{\clh_0})) = 0.
\]
Therefore, it follows from Corollary \ref{HiNew10} that $C(V_1|_{\clh_0}, V_2|_{\clh_0})$ has finite rank, and consequently, Corollary \ref{cor rank formula} implies
\[
\begin{split}
\text{rank }C(V_1|_{\clh_0}, V_2|_{\clh_0}) &= 2\text{rank } [(V_2|_{\clh_0})^*, V_1|_{\clh_0}] + \dim E_1(C(V_1|_{\clh_0}, V_2|_{\clh_0}))
\\
& \qquad - \dim E_{-1}(C(V_1|_{\clh_0}, V_2|_{\clh_0}))
\\
&=- \dim E_{-1}(C(V_1|_{\clh_0}, V_2|_{\clh_0})).
\end{split}
\]
Therefore,
\[
C(V_1|_{\clh_0}, V_2|_{\clh_0}) = 0,
\]
which completes the proof.
\end{proof}

Before we proceed to the final lemma of this section, we fix some notations. Set
\begin{equation}\label{eqn def of H0}
\clh_0 := \mathcal{H} \ominus (\bigoplus_{i = 1}^k \mathcal{H}_i),
\end{equation}
and for each $i=0, \ldots, k$, define
\[
(V_{1,i}, V_{2,i}):= (V_1|_{\clh_i}, V_2|_{\clh_i}).
\]
Clearly, $(V_{1,i}, V_{2,i})$ is a BCL pair on $\clh_i$, $i=0, 1, \ldots, k$. We have the following (see Definition \ref{def doubly comm} for the notion of shift-unitary pairs):

\begin{lemma}\label{slocinski}
$(V_{1,0}, V_{2,0})$ is a shift-unitary pair.
\end{lemma}
\begin{proof}
By Lemma \ref{kernel},
\[
C(V_{1,0}, V_{2,0}) = 0,
\]
which implies that $(V_{1,0}, V_{2,0})$ is a doubly commuting BCL pair on $\mathcal{H}_0$ \cite[Theorem 6.5]{MSS}. By \eqref{eqn Wold DC}, the Wold decomposition for doubly commuting pairs, there is a unique orthogonal decomposition of $\clh_0$ into $(V_{1,0}, V_{2,0})$-reducing subspaces
\[
\clh_0 = \clh_{uu} \oplus \clh_{us} \oplus \clh_{su} \oplus \clh_{ss},
\]
where $V_{1,0}$ on $\clh_{ij}$ is a shift (respectively, unitary) if $i = s$ (respectively, $i = u$) and $V_{2,0}$ on $\clh_{ij}$ is a shift (respectively, unitary) if $j = s$ (respectively, $j = u$). As $(V_{1,0}, V_{2,0})$ is a BCL pair, we must have that
\[
\clh_{uu} = \{0\}.
\]
By the construction of $\clh_{ss}$ \cite[Theorem 3.1]{JS}, it follows that
\[
\clh_{ss} = \bigoplus_{m, n \geq 0} V_{1,0}^m V_{2,0}^n \big(\ker V_{1,0}^* \cap \ker V_{2, 0}^*\big).
\]
On the other hand, Lemma \ref{intersection} implies
\[
E_1(C(V_{1, 0}, V_{2, 0})) = \ker V_{1,0}^* \cap \ker V_{2, 0}^*.
\]
Then
\[
\clh_{ss} = \bigoplus_{m, n \geq 0} V_{1,0}^m V_{2,0}^n E_1(C(V_{1, 0}, V_{2, 0})).
\]
Observe that
\[
E_1(C(V_{1, 0}, V_{2, 0})) = E_1(C(V_1, V_2)|_{\clh_0}) = \{0\},
\]
and hence
\[
\clh_{ss} = \{0\}.
\]
Thus
\[
\clh_{uu} = \clh_{ss} = \{0\}.
\]
Hence, $\clh_0 = \clh_{us} \oplus \clh_{su}$ and the result follows.
\end{proof}

Let us establish one terminology for the purpose of future reference.
 

\begin{definition}
Let $(V_1, V_2)$ on $\clh$ be a compact normal pair. Let $\clh_0$ be as in \eqref{eqn def of H0}. The shift-unitary part of $(V_1, V_2)$ is the pair $(V_{1, 0}, V_{2, 0})$ defined by
\[
(V_{1, 0}, V_{2, 0}) = (V_1|_{\clh_0}, V_2|_{\clh_0}).
\]
\end{definition}

 
 
 
\begin{remark}\label{e1zero}
Let $(V_1, V_2)$ on $\clh$ be a compact normal pair. Assume that $E_1 = \{0\}$. By Lemma \ref{contain}, we know
\[
[V_2^*, V_1] = 0,
\]
and so, by Corollary \ref{HiNew10}, $C(V_1,V_2)$ has finite rank, and consequently, it follows from Corollary \ref{cor rank formula} that
\[
C(V_1, V_2) = 0.
\]
Therefore, in this case, $\clh = \clh_0$ and consequently
\[
(V_1, V_2) = (V_{1, 0}, V_{2, 0}).
\]
An appeal to Lemma \ref{slocinski} immediately yields that $(V_1, V_2)$ on $\clh$ is a shift-unitary pair.
\end{remark}

The following theorem highlights all of the results achieved so far in this section:

\begin{theorem}\label{summ_1}
Let $(V_1, V_2)$ be a compact normal pair on $\clh$. Define
\[
k:= \text{dim} E_1(C(V_1, V_2)) \in [0, \infty].
\]
Then the following holds:
\begin{enumerate}
\item There exist $k+1$ closed $(V_1, V_2)$-reducing subspaces $\{\clh_j\}_{j=0}^k$ such that
\[
\clh = \bigoplus_{j=0}^k \clh_j,
\]
where
\[
\clh_j:= \overline{\text{span}}\{V_1^m V_2^n f_j: m, n \geq 0\} \qquad (j=1,\ldots, k),
\]
and
\[
\clh_0 = \clh \ominus (\bigoplus_{j=1}^k \clh_j),
\]
and $\{f_i\}_{i=1}^k$ is an orthonormal basis of $E_1(C(V_1, V_2))$ consisting of eigenvectors of the cross-commutator $[V_2^*, V_1]$.
\item $(V_{1,i}, V_{2,i})$ on $\clh_i$ is irreducible, where
\[
(V_{1,i}, V_{2,i}):= (V_1|_{\clh_i}, V_2|_{\clh_i}),
\]
for all $i=1, \ldots, k$.
\item $(V_{1, 0}, V_{2, 0})$ on $\clh_0$ is a shift-unitary pair, where
\[
(V_{1, 0}, V_{2, 0}) = (V_1|_{\clh_0}, V_2|_{\clh_0}).
\]
\end{enumerate}
\end{theorem}

We continue with the assumptions and conclusion of the previous theorem. Our aim is to analyze the structure of the irreducible pair $(V_{1,i}, V_{2,i})$ on $\mathcal{H}_i$ for $i = 1, 2, \cdots, k$. Note that the structure of $(V_{1,0}, V_{2,0})$ is clear from part (3) of the preceding theorem.

We fix an $i \in \{1, \cdots, k\}$, and set
\[
E_{1,i} := E_1 \big(C(V_{1,i}, V_{2,i})\big).
\]
It is clear from the definition of the space $\mathcal{H}_i$ and Lemma \ref{intersection} that
\begin{equation}\label{ranke1}
E_{1,i}=  \ker V_{1,i}^* \cap \ker V_{2, i}^* = \mathbb{C} f_i.
\end{equation}
Then
\[
[V_{2, i}^*, V_{1,i}]|_{\clh_i \ominus E_{1,i}} = 0.
\]
Moreover, $[V_{2, i}^*, V_{1,i}] f_i = \lambda_i f_i$ implies that
\[
\text{rank}[V_{2, i}^*, V_{1,i}] =
\begin{cases}
0 & \mbox{if } \lambda_i = 0
\\
 1 & \mbox{if }  \lambda_i \neq 0.
\end{cases}
\]
We first consider the case when $\lambda_i = 0$. In this case, $[V_{2,i}^*, V_{1, i}] = 0$ and hence, $(V_{1,i}, V_{2,i})$ on $\mathcal{H}_i$ is doubly commuting. By \cite[Theorem 6.5]{MSS}, $C(V_{1,i}, V_{2,i}) \geq 0$ and therefore
\[
\dim E_{-1} \big(C(V_{1,i}, V_{2,i})\big) = 0.
\]
Consequently, by the second part of Theorem \ref{index}, we have
\[
\text{rank}C(V_{1,i}, V_{2,i}) = \dim E_{1,i} = 1.
\]
Then
\[
\ker \big(C(V_{1,i}, V_{2,i})\big)^\perp = E_{1,i},
\]
and hence
\[
\mathcal{W}_i := \ker (V_{1,i} V_{2,i})^* = E_{1,i} \oplus (\mathcal{W}_i \ominus E_{1,i}).
\]
With respect to this decomposition of $\mathcal{W}_i$, we write
\[
C(V_{1,i}, V_{2,i})|_{\mathcal{W}_i} = \begin{bmatrix}
	I_{E_{1,i}} &    \\
	& 0
\end{bmatrix}.
\]
As $(V_{1,i}, V_{2,i})$ is an irreducible doubly commuting BCL pair on $\mathcal{H}_i$ with nonzero defect operator, it follows directly from the construction of Wold decomposition for doubly commuting pairs \cite[Theorem 3.1]{JS} that
\[
\mathcal{H}_i = \bigoplus_{m, n \geq 0} V_{1,i}^m V_{2,i}^n E_{1,i},
\]
and $V_{1,i}$ and $V_{2,i}$ are unilateral shifts. More specifically
\[
(V_{1,i}, V_{2,i}) \cong (M_z, M_w) \text{ on } H^2(\mathbb{D}^2).
\]
Now we consider the case when $\lambda_i \neq 0$. Since $\text{rank}[V_{2, i}^*, V_{1,i}] = 1$, an appeal to Corollary \ref{HiNew10} yields that
\[
\text{rank} C(V_{1,i}, V_{2,i})< \infty,
\]
and consequently, Corollary \ref{cor rank formula} implies
\begin{equation}\label{rankeqn}
\text{rank}C(V_{1,i}, V_{2,i}) = 3 - \dim E_{-1, i},
\end{equation}
where
\[
E_{-1, i}: = E_{-1}(C(V_{1,i}, V_{2,i})).
\]
Clearly
\[
\dim E_{-1, i} \leq 3.
\]
If $\dim E_{-1, i} = 3$, then \eqref{rankeqn} implies $C(V_{1,i}, V_{2,i}) = 0$, a contradiction. Therefore,
\[
\dim E_{-1, i} \in \{0,1,2\}.
\]
We now consider three separate cases:
\begin{enumerate}
\item[(i)] Suppose $\dim E_{-1, i} = 2$. Since $\dim E_{1,i} = 1$ (see \eqref{ranke1}), it follows that
\[
\text{rank} C(V_{1,i}, V_{2,i}) \geq 3.
\]
On the other hand, it follows from \eqref{rankeqn} that, in this case
\[
\text{rank}C(V_{1,i}, V_{2,i}) = 1,
\]
a contradiction. Therefore
\[
\dim E_{-1, i} \neq 2.
\]
\item[(ii)] Suppose $\dim E_{-1, i} = 1$. We know, by \eqref{rankeqn}, that
\[
\text{rank}C(V_{1,i}, V_{2,i}) = 2.
\]
Thus, $(V_{1,i}, V_{2,i})$ is an irreducible $2$-finite pair. This class of pairs was classified earlier in Theorem \ref{rank2}, leading us to conclude the existence of a unimodular constant $\alpha$ such that
\[
\sigma([V_{2,i}^*, V_{1, i}]) \setminus \{0\} = \{\alpha\}.
\]
\item[(iii)] Finally, if $\dim E_{-1, i} = 0$, then \eqref{rankeqn} again implies that
\[
\text{rank}C(V_{1,i}, V_{2,i}) = 3.
\]
Therefore, in this case, $(V_{1,i}, V_{2,i})$ is an irreducible $3$-finite pair on $\clh_i$, which was classified in Theorem \ref{irrfinal}. Consequently, there exist $\lambda \in (0, 1)$ and unimodular constant $\gamma$ such that
\[
\sigma(C(V_{1,i}, V_{2,i})) \cap (0, 1) = \{\lambda\},
\]
and
\[
\sigma([V_{2,i}^*, V_{1, i}]) \setminus \{0\} = \{\gamma \lambda\}.
\]
\end{enumerate}

Summarizing the foregoing discussion, we have:

\begin{proposition}\label{summ_2}
In the setting of Theorem \ref{summ_1}, fix $i \in \{1, \ldots, k\}$. Then
\[
\text{rank}[V_{2, i}^*, V_{1,i}] \in \{0,1\},
\]
and we have the following:

\begin{enumerate}
\item If $\text{rank}[V_{2, i}^*, V_{1,i}] = 0$, then
\[
(V_{1,i}, V_{2,i}) \cong (M_z, M_w) \text{ on } H^2(\D^2).
\]
\item If $\text{rank}[V_{2, i}^*, V_{1,i}] = 1$, then
\[
\text{rank} C(V_{1,i}, V_{2,i}) \in \{2,3\}.
\]
\begin{itemize}
\item[(a)] If $\text{rank} C(V_{1,i}, V_{2,i}) = 2$, then $(V_{1,i}, V_{2,i})$ is an irreducible $2$-finite pair, and
\[
\sigma(C(V_{1,i}, V_{2,i})) \setminus \{0\} = \{\pm 1\}, \text{ and } \sigma([V_{2,i}^*, V_{1, i}]) \setminus \{0\} = \{\alpha\}
\]
for some unimodular constant $\alpha$.
	
\item[(b)] If $\text{rank} C(V_{1,i}, V_{2,i}) = 3$, then $(V_{1,i}, V_{2,i})$ is an irreducible $3$-finite pair, and there exist $\lambda \in (0, 1)$, and a unimodular constant $\gamma$ such that
\[
\sigma(C(V_{1,i}, V_{2,i})) \cap (0, 1) = \{\lambda\},
\]
and
\[
\sigma([V_{2,i}^*, V_{1, i}]) \setminus \{0\} = \{\lambda \gamma\}.
\]
	
\end{itemize}
\end{enumerate}
\end{proposition}

In view of this, the main result concerning a complete description of compact normal pairs can now be stated. In essence, the description is a summary of all the major outcomes in this paper so far, specifically Theorem \ref{irrfinal}, and Theorem \ref{rank2}, Theorem \ref{summ_1}, and Proposition \ref{summ_2}.

\begin{theorem}\label{main}
Let $(V_1, V_2)$ be a compact normal pair on $\clh$. Define
\[
k: = \dim E_1 \big(C(V_1, V_2)\big) \in [0, \infty].
\]
Then there exist $k+1$ closed $(V_1, V_2)$-reducing subspaces $\{\clh_i\}_{i=0}^k$ of $\clh$ such that
\[
\clh = \bigoplus_{i=0}^k \clh_i.
\]
Set
\[
(V_{1,i}, V_{2,i}) = (V_1|_{\clh_i}, V_2|_{\clh_i}),
\]
for all $i=0, 1, \ldots, k$. Then, we have the following:
\begin{enumerate}
\item $(V_{1,0}, V_{2,0})$ on $\clh_0$ is a shift-unitary pair.
\item For each $i=1, \ldots, k$, the pair $(V_{1,i}, V_{2,i})$ on $\clh_i$ is irreducible and is unitarily equivalent to one of the following three pairs:
\begin{enumerate}
\item $(M_z, M_w)$ on $H^2(\D^2)$.
\item $(M_z, {\alpha} M_z)$ on $H^2(\D)$ for some unimodular constant $\alpha$.
\item $(\gamma M_z|_{\cls_{\lambda}}, M_w|_{\cls_{\lambda}})$ on $\cls_{\lambda}$ where $\cls_{\lambda}$ is the invariant subspace of $H^2(\D^2)$ given by
\[
\cls_{\lambda} = \vp \Big(H^2(\mathbb{D}^2) \bigoplus \Big(\bigoplus_{j = 0}^\infty z^j \text{span} \Big\{ \frac{\bar{w}}{1 - \lambda z \bar{w}}\Big\} \Big) \Big),
\]
for some $\lambda \in (0,1)$, unimodular constant $\gamma$ and inner function $\vp \in H^\infty(\D^2)$.
\end{enumerate}
\end{enumerate}
\end{theorem}

In the following section, we use this result to explain a complete set of unitary invariants for compact normal pairs. The discourse that precedes Proposition \ref{summ_2} also gives, in particular, that
\[
\text{rank} C(V_{1,i}, V_{2,i}) = 1, 2, \text{ or } 3,
\]
for all $i=1, \ldots, k$. This yields the complete list of irreducible $n$-finite pairs. More specifically:

\begin{coro}\label{cor irred n finite}
An irreducible $n$-finite pair is either $1$-finite, $2$-finite, or $3$-finite.
\end{coro}

Keep in mind that $n$-finite pairs with $n>3$ still exist, but irreducible $n$-finite pairs for $n = 1, 2$, and $3$ will build them up. In particular, $n$-finite pairs, $n>3$, are always reducible.  

\section{Complete unitary invariants}\label{sec unitary inv}

In this section, we analyze the main results in terms of unitary invariants. First, we note, as already pointed out in the results obtained so far (see Theorem \ref{irrfinal}, Theorem \ref{rank2}, Theorem \ref{summ_1}, and Proposition \ref{summ_2}), that the decomposition in Theorem \ref{main} is unique (up to unitary equivalence) and canonical. In order to be more specific, let us continue with the assumptions and outcomes of Theorem \ref{main}. Recall that $\{f_1, \ldots, f_k\}$ was assumed to be an orthonormal basis of $E_1\big(C(V_1, V_2)\big)$ consisting of eigen vectors of $[V_2^*, V_1]$, where
\[
k = \dim E_1(C(V_1, V_2)) \in [0, \infty].
\]
Then
\[
\mathcal{H}_i := \overline{\text{span}}\{V_1^m V_2^n f_i : m, n \geq 0\},
\]
reduces $(V_1, V_2)$ for all $i=1, \ldots, k$. Moreover, we have the remaining space
\[
\clh_0 := \clh \ominus (\bigoplus_{i = 1}^k \clh_i).
\]
Recall
\[
(V_{1,i}, V_{2,i}) := (V_1|_{\clh_i}, V_2|_{\clh_i}),
\]
for all $i=0, 1, \ldots, k$, where $(V_{1,0}, V_{2,0})$ on $\clh_0$ is a shift-unitary pair, and for each $i \in \{1, \ldots, k\}$, the pair $(V_{1,i}, V_{2,i})$ on $\clh_i$ satisfies the following properties:
	
\begin{enumerate}
\item $(V_{1,i}, V_{2,i})$ is an irreducible $1$-finite pair if and only if
\[
(V_{1,i}, V_{2,i}) \cong (M_z, M_w) \text{ on } H^2(\mathbb{D}^2).
\]
Moreover, in this case
\[
\text{rank}[V_{2,i}^*, V_{1,i}] = 0.
\]
\item $(V_{1,i}, V_{2,i})$ is an irreducible $2$-finite pair if and only if there exists unimodular constant $\alpha$ such that
\[
(V_{1,i}, V_{2,i}) \cong (M_z, \alpha M_z) \text{ on } H^2(\D).
\]
Moreover, in this case
\[
\text{rank}[V_{2,i}^*, V_{1,i}] = 1,
\]
and
\[
\sigma([V_{2,i}^*, V_{1,i}]) \setminus \{0\} = \{\bar{\alpha}\}.
\]
\item $(V_{1,i}, V_{2,i})$ is an irreducible $3$-finite pair if and only if there exist $\lambda \in (0, 1)$ and a unimodular constant $\gamma$ such that
\[
(V_{1,i}, V_{2,i}) \cong (\gamma M_z|_{\cls_{\lambda}}, M_w|_{\cls_{\lambda}}) \text{ on } \cls_{\lambda}.
\]
Moreover, in this case
\[
\text{rank}[V_{2,i}^*, V_{1,i}] = 1,
\]
and
\[
\sigma([V_{2,i}^*, V_{1,i}]) \setminus \{0\} = \{\lambda \gamma\}.
\]
\end{enumerate}

Now we turn to the problem of computing complete unitary invariants. We fix a compact normal pair $(V_1, V_2)$ on a Hilbert space $\clh$. We adhere to the conclusion and the identical notation as presented in the preceding discussion and Theorem \ref{main}. Recall the notation that
\[
k = \dim E_1(C(V_1, V_2)).
\]
Suppose
\[
k > 0.
\]
We construct a sequence
\[
\alpha_{(V_1,V_2)} = \{\alpha_i\}_{i=1}^k \subseteq \mathbb{C},
\]
as follows:
\[
\alpha_i :=
\begin{cases}
0 & \mbox{if } \text{ rank}[V_{2, i}^*, V_{1, i}] = 0 \\
\sigma([V_{2, i}^*, V_{1, i}]) \setminus \{0\} & \mbox{if} \text{ rank} [V_{2, i}^*, V_{1, i}] = 1.
\end{cases}
\]
If
\[
k = 0,
\]
then we define
\[
\alpha_{(V_1,V_2)} = \text{ empty sequence}.
\]
The \textit{cardinality} of the sequence $\{\alpha_i\}$ is the number $k = \dim E_1(C(V_1, V_2))$.

\begin{definition}
The sequence
\[
\alpha_{(V_1,V_2)} = \{\alpha_i\}_{i=1}^k \subseteq \mathbb{C},
\]
is referred to as the fundamental sequence associated to the isometric pair $(V_1, V_2)$.
\end{definition}

The term ``fundamental sequence'' finds its rationale in the fact that this sequence is determined by the action of the cross-commutator on the fundamental building blocks consisting of irreducible $n$-finite pairs, $n=1,2,3$. In terms of fundamental sequence, the discussion at the beginning of this section results in the following:
\begin{enumerate}
\item $\alpha_i = 0$ if and only if
\[
(V_{1,i}, V_{2,i}) \cong (M_z, M_w) \text{ on } H^2(\mathbb{D}^2).
\]
\item $|\alpha_i| = 1$ if and only if
\[
(V_{1,i}, V_{2,i}) \cong (M_z, \bar{\alpha_i} M_z) \text{ on } H^2(\D).
\]
\item $0 < |\alpha_i|< 1$ if and only if
\[
(V_{1,i}, V_{2,i}) \text{ on } \clh_i \cong (\gamma M_z|_{\cls_{\lambda}}, M_w|_{\cls_{\lambda}}) \text{ on } H^2(\D^2),
\]
where $\lambda \in (0,1)$, and $\gamma$ is a unimodular constant such that $\lambda\gamma = \alpha_i$.
\end{enumerate}
Therefore, we have the following:
\[
(V_1|_{\clh_0^\perp}, V_2|_{\clh_0^\perp})  \cong M_1 \oplus M_2 \oplus M_3,
\]
where
\[
M_1 = \bigoplus_{\{i: \alpha_i = 0 \}} (M_z, M_w),
\]
and
\[
M_2= \bigoplus_{\{i: |\alpha_i| = 1\}}(M_z, \bar{\alpha_i} M_z),
\]
and
\[
M_3 = \bigoplus_{\substack{\{i: 0 < |\alpha_i| < 1 \text{ with } \alpha_i = \lambda_i \gamma_i, \\  \lambda_i \in (0, 1),  |\gamma_i| = 1\}}}(\gamma_i M_z|_{\cls_{\lambda_i}}, M_w|_{\cls_{\lambda_i}}).
\]
In summary of the above discussion, the shift-unitary part, along with the fundamental sequence, serves as a complete unitary invariant for compact normal pairs. More formally:

\begin{theorem}\label{them complete unit inv}
Let $(V_1, V_2)$ be a compact normal pair on $\clh$. Let
\[
\alpha_{(V_1,V_2)} = \{\alpha_i\}_{i=1}^k,
\]
denote the fundamental sequence associated to $(V_1, V_2)$ with cardinality $k \in [0, \infty]$. Also, let $(V_{1,0}, V_{2,0})$ on $\clh_0$ denote the shift-unitary part of $(V_1,V_2)$.

(i) Then
\[
(V_1|_{\clh_0^\perp}, V_2|_{\clh_0^\perp})  \cong M_1 \oplus M_2 \oplus M_3,
\]
where $M_i$, $i=1,2,3$, are defined as above.

(ii) Let $(\tilde{V_1}, \tilde{V_2})$ on $\tilde{\clh}$ be another compact normal pair. Suppose $\tilde \alpha_{(\tilde V_1, \tilde V_2)} = \{\tilde \alpha_i\}_{i=1}^{\tilde k}$ is the associated fundamental sequence with cardinality $\tilde{k} \in [0, \infty]$. Assume that $(\tilde V_{1,0}, \tilde V_{2,0})$ on $\tilde \clh_0$ denotes the shift-unitary part of $(\tilde V_1, \tilde V_2)$. Then the following are equivalent:
\begin{enumerate}
\item $(V_1, V_2) \cong (\tilde{V_1}, \tilde{V_2})$.
\item $(V_{1,0}, V_{2,0}) \cong (\tilde V_{1,0}, \tilde V_{2,0})$, and $[V_2^*, V_1]|_{{\clh_0}^\perp} \cong [\tilde{V_2}^*, \tilde{V_1}]|_{{\tilde{\clh_0}^\perp}}$.
\item $(V_{1,0}, V_{2,0}) \cong (\tilde V_{1,0}, \tilde V_{2,0})$, $k = \tilde{k}$, and there exists a permutation $\sigma$ of $\{1, 2, \cdots, k\}$ such that
\[
\alpha_i = \tilde \alpha_{\sigma(i)} \qquad (i = 1, 2, \cdots, k).
\]
\end{enumerate}
\end{theorem}

Our aim in the following remark is to examine the unitary equivalence of the shift-unitary parts of compact normal pairs (more specifically, the first part of (3) of the above theorem). We assert that this part is the simplest of the entire equivalence problem.

\begin{remark}\label{rem shift unit}
Let $(V_1, V_2)$ be a shift-unitary pair on a Hilbert space $\clh$. First we write the two summands of reducing subspaces as (see Definition \ref{def doubly comm})
\[
\clh = \clh_{us} \oplus \clh_{su}.
\]
By symmetry, it is now evident to explore unitary equivalence of $(V_1|_{\clh_{su}}, V_2|_{\clh_{su}})$. More generally, we consider a commuting pair $(W_1, W_2)$ on a Hilbert space $\clk$ such that $W_1$ is a shift and $W_2$ is a unitary. Therefore, there exist a Hilbert space $\clw$ and a unitary $U: \clk \rightarrow H^2_{\clw}(\D)$ such that (refer to the discourse given at the outset of Section \ref{sec preliminaries})
\[
U W_1 = M_z U.
\]
Since $W_2$ is a unitary commuting and $*$-commuting with the shift $W_1$, it follows that a constant function yields the analytic Toeplitz representation of $W_2$ on $H^2_{\clw}(\D)$. In other words, there exists a unitary operator $W \in \clb(\clw)$ such that
\[
U W_2 = (I_{H^2(\D)} \otimes W) U.
\]
Therefore, the pair $\{W, \clw\}$ is a complete set of unitary invariants for shift-unitary pairs of the above type.
\end{remark}

Given the outcomes of this paper, we are compelled to pose the following natural question:

\begin{question}\label{quest 1}
Classify isometric pairs $(V_1, V_2)$ acting on Hilbert spaces such that
\[
[V_2^*, V_1] = \text{compact}.
\]
\end{question}

It perhaps necessitates different methodologies. As far as the present methodology is concerned, our approach involved identifying a complete list of irreducible compact normal pairs and thereafter representing a typical compact normal pair as a direct sum of them. In the present paper, the irreducible $n$-finite pairs, where $n$ is equal to $1, 2$, and $3$, as well as the shift-unitary pairs, have served as distinguished building blocks.

An essential step in answering the question might include identifying certain irreducible pairs of isometries that satisfy the above compactness condition. A further strategy could involve incorporating the defect operator of isometric pairs, as was demonstrated in this paper. Defect operators are indeed very significant; however, they alone do not provide substantial information about pairs. For instance, consider the pairs $(M_z, M_z)$ and $(M_z, \gamma M_z)$ on the Hardy space $H^2(\mathbb{D})$, where
\[
\gamma \in \mathbb{T} \setminus \{1\},
\]
is a fixed scalar. An easy computation yields
\[
C(M_z, M_z) = C(M_z, \gamma M_z).
\]
However, it is easy to see that $(M_z, M_z)$ and $(M_z, \gamma M_z)$ are not jointly unitarily equivalent. As observed earlier, these are the examples of 2-finite pairs. Some of the results of the present paper could also be helpful in answering the above question. For instance, Theorem \ref{index} is true for all isometric pairs.

Question \ref{quest 1} has an $n$-variable analogy, $n > 2$. It is important to note, nevertheless, that operator and function theory depart considerably when $n$ rises from two to three or even larger. 

\begin{question}\label{quest 2}
Classify $n$-tuples, $n > 2$, of commuting isometries $(V_1, \ldots, V_n)$ acting on Hilbert spaces such that
\[
[V_i^*, V_j] = \text{compact},
\]
or
\[
[V_i^*, V_j] = \text{compact and normal},
\]
for all $i\neq j$.
\end{question}

\vspace{0.2in}

\noindent\textbf{Acknowledgement:} The first named author is supported by the MATRICS research grant, SERB (File Number: MTR/2022/000921). The research of the second and third named authors is supported in part by TARE (TAR/2022/000063) by SERB, Department of Science \& Technology (DST), Government of India. The fourth-named author's research was supported by the NBHM Postdoctoral Fellowship (Order No. 0204/10/(24)/2023/R\&D-II/2802) during his postdoctoral tenure at IIT Goa.

\end{document}